\documentclass[11pt]{article}

\usepackage{geometry}
\usepackage{amsthm}
\usepackage{amsmath}
\usepackage{amssymb}
\usepackage{mathtools}
\usepackage{mathrsfs}
\usepackage{subfigure}
\usepackage{multicol}
\usepackage{tikz}

\newcommand{\bb}{\mathbb}
\newcommand{\mc}{\mathcal}
\newcommand{\ms}{\mathscr}
\newcommand{\abs}[1]{\lvert #1 \rvert}

\newcommand{\minus}{\,\backslash\,}

\DeclareMathOperator{\conv}{conv}
\DeclareMathOperator{\link}{link}
\DeclareMathOperator{\type}{type}

\theoremstyle{definition}

\newtheorem{thm}{Theorem}[section]
\newtheorem{prop}[thm]{Proposition}
\newtheorem{lem}[thm]{Lemma}

\newtheorem{defn}[thm]{Definition}
\newtheorem{exam}[thm]{Example}

\numberwithin{equation}{section}

\title{A zonotope and a product of two simplices with disconnected flip graphs}
\author{Gaku Liu}

\begin{document}
\maketitle

\begin{abstract}
We give an example of a three-dimensional zonotope whose set of tight zonotopal tilings is not connected by flips. Using this, we show that the set of triangulations of $\Delta^4 \times \Delta^n$ is not connected by flips for large $n$. Our proof makes use of a non-explicit probabilistic construction.
\end{abstract}

\section{Introduction}

We consider the poset $\mc P$ of polyhedral subdivisions of a polytope $P$ or zonotopal tilings of a zonotope $Z$, ordered by refinement. This poset is called the \emph{Baues poset} of $P$ or $Z$. The minimal elements of this poset are, respectively, the triangulations of $P$ or the tight zonotopal tilings of $Z$. Two minimal elements of $\mc P$ are connected by a \emph{flip} if there is an element of $\mc P$ whose only proper refinements are these two minimal elements. The \emph{flip graph} of $P$ or $Z$ is the graph whose vertices are minimal elements of $\mc P$ and whose edges are flips.

For zonotopal tilings, the flip graph is known to be connected for cyclic zonotopes \cite{Zie93} or if the zonotope has dimension two \cite{Ken93}. Our first result is the first example, to the author's knowledge, of a zonotope whose flip graph is not connected. This answers a question of Reiner in \cite{Rei99}. Our example is a three-dimensional permutohedron with many copies of each of its generating vectors. The number of copies of each vector is determined by a probabilistic argument to be around 100.

Using a related construction, we also show that the flip graph of triangulations of the product of two simplices is not generally connected. Santos \cite{San00} gave an example of a polytope whose flip graph of triangulations is not connected. However, the case when $P$ is product of two simplices remained of special interest due to the appearances of these triangulations in various branches of mathematics; see \cite[Chapter 6.2]{DRS} for an overview. Santos \cite{San05} proved that the flip graph of $\Delta^2 \times \Delta^n$ is connected for all $n$, and the author \cite{Liu16} proved that the flip graph of $\Delta^3 \times \Delta^n$ is connected for all $n$. However, we show that the flip graph of $\Delta^4 \times \Delta^n$ is not connected for $n \approx 4 \cdot 10^4$.

Connectivity of flip graphs is related to the \emph{generalized Baues problem}, which concerns the topology of the poset $\mc P$. Specifically, the problem asks if the order complex of $\mc P$ minus its maximal element is homotopy equivalent to a sphere. The question was formulated by Billera et al.\ in \cite{BKS94} and is motivated by the fact that when one restricts $\mc P$ to \emph{coherent} subdivisons or zonotopal tilings, the resulting poset is isomorphic to the face lattice of a certain polytope, called the \emph{fiber polytope} \cite{BS92}. See \cite{Rei99} for an overview.

In general, the answers to the generalized Baues problem and the question of connectivity of flip graphs do not imply each other. However, there are situations where the answer to one can be used to answer the other. The generalized Baues problem for triangulations was answered in the negative by Santos \cite{San06} by using a point set in general position with disconnected flip graph. The problem for zonotopal tilings remains open.

The paper is organized as follows. Section 2 reviews triangulations and the product of two simplices. Section 3 reviews mixed subdivisions, zonotopal tilings, and the Cayley trick. Section 4 constructs our zonotope and proves that its flip graph is not connected. Section 5 proves that the flip graph of $\Delta^4 \times \Delta^n$ is not connected. Section 6 is an appendix proving several propositions used in the paper.

\section{Triangulations and the product of two simplices}

We begin with a quick overview of triangulations, flips, and the product of two simplices. We refer to De Loera et al.\ \cite{DRS} for a more comprehensive treatment.

\subsection{Subdivisions and triangulations}

Throughout this section, let $A \subset \bb R^m$ be a finite set of points. A \emph{cell} of $A$ is a subset of $A$. A \emph{simplex} is a cell which is affinely independent. A \emph{face} of a cell $C$ is a subset $F \subseteq C$ such that there exists a linear functional $\phi \in (\bb R)^m$ such that $F$ is the set of all points which minimize $\phi$ on $C$. For any cell $C$, let $\conv(C)$ denote the convex hull of $C$.

\begin{defn}
A \emph{polyhedral subdivision}, or \emph{subdivision}, of $A$ is a collection $\ms S$ of cells of $A$ such that
\begin{enumerate}
\item If $C \in \ms S$ and $F$ is a face of $C$, then $F \in \ms S$.
\item If $C$, $C' \in \ms S$, then $\conv(C) \cap \conv(C') = \conv(F)$ where $F$ is a face of $C$ and $C'$.
\item $\bigcup_{C \in \ms S} \conv(C) = \conv(A)$.
\end{enumerate}
The subdivision $\{A\}$ is the \emph{trivial subdivision}. A subdivision all of whose elements are simplices is a \emph{triangulation}.
\end{defn}

If $\ms S$ is a subdivision of $A$ and $F$ is a face of $A$, then $\ms S$ induces a subdivision $\ms S[F]$ of $F$ by
\[
\ms S[F] := \{ C \in \ms S : C \subseteq F \}. 
\]

For subdivisions $\ms S$, $\ms S'$, we say that $\ms S$ is a \emph{refinement} of $\ms S'$ if every element of $\ms S$ is a subset of an element of $\ms S'$. Refinement gives a poset structure on the set of all subdivisions of $A$. The maximal element of this poset is the trivial subdivision and the minimal elements are the triangulations.

\subsection{Flips}

As stated in the introduction, two triangulations are connected by a flip if there is a subdivision whose only proper refinements are these two triangulations. We will now give an equivalent definition of a flip which will be easier to use.

A \emph{circuit} is a minimal affinely dependent subset of $\bb R^d$. If $X = \{x_1,\dotsc,x_k\}$ is a circuit, then the elements of $X$ satisfy an affine dependence equation
\[
\sum_{i=1}^k \lambda_i x_i = 0
\]
where $\lambda_i \in \bb R \minus \{0\}$ for all $i$, $\sum_i \lambda_i = 0$, and the equation is unique up to multiplication by a constant. This gives a unique partition $X = X^+ \cup X^-$ of $X$ given by $X^+ = \{x_i : \lambda_i > 0\}$ and $X^- = \{x_i : \lambda_i < 0\}$. We will write $X = (X^+, X^-)$ to denote a choice of which part we call $X^+$ and which we call $X^-$.

A circuit $X = (X^+, X^-)$ has exactly two non-trivial subdivisions, which are the following triangulations:
\[
\ms T_X^+ := \{ \sigma \subseteq X : \sigma \not\supseteq X^+ \} \qquad \ms T_X^- := \{ \sigma \subseteq X : \sigma \not\supseteq X^- \}.
\]

Given a subdivision $\ms S$ and a cell $C \in \ms S$, we define the \emph{link} of $C$ in $\ms S$ as
\[
\link_{\ms S}(C) := \{ C' \in \ms S : C \cap C' = \emptyset,\, C \cup C' \in \ms S \}.
\]
We can now state the definition of a flip, in the form of a proposition.

\begin{prop}[Santos \cite{San02}] \label{flipdef}
Let $\ms T$ be a triangulation of $A$. Suppose there is a circuit $X = (X^+,X^-)$ contained in $A$ such that
\begin{enumerate}
\item $\ms T_X^+ \subseteq \ms T$.
\item All maximal simplices of $\ms T_X^+$ have the same link $\ms L$ in $\ms T$.
\end{enumerate}
Then the collection
\[
\ms T' := \ms T \minus \{ \rho \cup \sigma : \rho \in \ms L, \sigma \in \ms T_X^+ \} \cup \{ \rho \cup \sigma : \rho \in \ms L, \sigma \in \ms T_X^- \}
\]
is a triangulation of $A$. We say that $\ms T$ has a \emph{flip} supported on $(X^+,X^-)$, and that $\ms T'$ is the result of applying this flip to $\ms T$.
\end{prop}

\subsection{Regular subdivisions}

Let $A \subseteq \bb R^m$ be as before. Let $\omega : A \to \bb R$ be any function. For a cell $C \subseteq A$, we define the \emph{lift} of $C$ to be the set $C^\omega \subset \bb R^m \times \bb R$ given by
\[
C^\omega := \{ (x, \omega(x)) : x \in C \}.
\]
We call a subset $F \subseteq A^\omega$ a \emph{lower face} of $A^\omega$ if either $F$ is empty or there is a linear functional $\phi \in (\bb R^m \times \bb R)^\ast$ such that $\phi(0,1) > 0$ and $F$ is the set of all points which minimize $\phi$ on $A^\omega$. Then the collection of all $C \subseteq A$ such that $C^\omega$ is a lower face of $A^\omega$ is a subdivision of $A$. We call this the \emph{regular} subdivision of $A$ with respect to $\omega$, and denote it by $\ms S_A^\omega$.

If $F$ is a face of $A$, then $\ms S_A^\omega[F] = \ms S_F^{\omega \vert_F}$.

Both triangulations of a circuit are regular, as stated below. 
\begin{prop} \label{circuittiling}
Suppose $\sum_{i=1}^k \lambda_i x_i = 0$ is the affine dependence equation for a circuit $X = (X^+,X^-)$ with $X^+ = \{x_i : \lambda_i > 0\}$ and $X^- = \{x_i : \lambda_i < 0\}$. Let $\omega : X \to \bb R$ be a function. Then
\[
\ms S_X^\omega = 
\begin{dcases*}
\ms T_X^+ & if \ $\sum_{i=1}^k \lambda_i \omega(x_i) > 0$ \\
\ms T_X^- & if \ $\sum_{i=1}^k \lambda_i \omega(x_i) < 0$.
\end{dcases*}
\]
\end{prop}

\subsection{The product of two simplices}

We now consider $\Delta^{m-1} \times \Delta^{n-1}$, the product of two simplices of dimensions $m-1$ and $n-1$. Following the conventions of the previous section, we will understand $\Delta^{m-1} \times \Delta^{n-1}$ to mean the set of vertices of $\Delta^{m-1} \times \Delta^{n-1}$ rather than the polytope itself.

Let $\Delta^{m-1} := \{e_1, \dotsc, e_m\}$ be the standard basis for $\bb R^m$ and $\Delta^{n-1} := \{f_1, \dotsc, f_n\}$ be the standard basis for $\bb R^n$. We embed $\Delta^{m-1} \times \Delta^{n-1}$ in $\bb R^m \times \bb R^n$ by
\[
\Delta^{m-1} \times \Delta^{n-1} := \{ (e_i, f_j) : i \in [m], j \in [n] \}.
\]

Let $G := K_{m,n}$ be the complete bipartite graph with vertex set $\Delta^{m-1} \cup \Delta^{n-1}$ and edge set $\{ e_if_j : i \in [m], j \in [n]\}$. We have a bijection $(e_i,f_j) \mapsto e_if_j$ between $\Delta^{m-1} \cup \Delta^{n-1}$ and the edge set of $G$. For each cell $C \subseteq \Delta^{m-1} \times \Delta^{n-1}$, let $G(C)$ be the minimal subgraph of $G$ with edge set $\{ e_if_j : (e_i,f_j) \in C\}$. Then $C$ is a simplex if and only if $G(C)$ is acyclic, and $C$ is a circuit if and only if $G(C)$ is a cycle. If $C$ is a circuit, then $G(C)$ alternates between edges corresponding to positive and negative elements of the circuit.

\section{Zonotopal tilings}

We now define zonotopal tilings in terms of mixed subdivisions. We review the Cayley trick which shows that mixed subdivisions are special cases of subdivisions. The information in this section was developed in \cite{Stu94}, \cite{HRS00}, and \cite{San05}.

\subsection{Mixed subdivisions} \label{sec:mixedsub}

Let $A_1$, \dots, $A_n$ be finite subsets of $\bb R^m$. The \emph{Minkowski sum} of $A_1$, \dots, $A_n$ is the set of points
\[
\sum A_i = A_1 + \dotsb + A_n := \{ x_1 + \dotsb + x_n : x_i \in A_i \text{ for all } i \}.
\]
In this paper, we want the phrase ``$\sum A_i$'' to identify a set of points but also retain the information of what $A_1$, \dots, $A_n$ are. In other words, $\sum A_i$ will formally mean an ordered tuple $(A_1,\dots,A_n)$ but by abuse of notation will also refer to the Minkowski sum.

A \emph{mixed cell} of $\sum A_i$ is a set $\sum B_i$ where $B_i$ is a cell of $A_i$ for all $i$. A mixed cell is \emph{fine} if all the $B_i$ are simplices and lie in independent affine subspaces. A \emph{face} of a mixed cell $\sum B_i$ is a mixed cell $\sum F_i$ of $\sum B_i$ such that there exists a linear functional $\phi \in (\bb R^m)^\ast$ such that for all $i$, either $F_i = \emptyset$ or $F_i$ is the set of all points which minimize $\phi$ on $B_i$.

\begin{defn}
A \emph{mixed subdivision} of $\sum A_i$ is a collection $\ms S$ of mixed cells of $\sum A_i$ such that
\begin{enumerate}
\item If $\sum B_i \in \ms S$ and $\sum F_i$ is a face of $\sum B_i$, then $\sum F_i \in \ms S$.
\item If $\sum B_i$, $\sum B_i' \in \ms S$, then $\conv(\sum B_i) \cap \conv(\sum B_i') = \conv(\sum F_i)$ where $\sum F_i$ is a face of $\sum B_i$ and $\sum B_i'$.
\item $\bigcup_{\sum B_i \in \ms S} \conv(\sum B_i) = \conv(\sum A_i)$.
\end{enumerate}
A mixed subdivision is \emph{fine} if all of its elements are fine.
\end{defn}

For mixed subdivisions $\ms S$, $\ms S'$, we say that $\ms S$ is a \emph{refinement} of $\ms S'$ if every element of $\ms S$ is a mixed cell of an element of $\ms S'$. Refinement gives a poset structure on the set of mixed subdivisions of $\sum A_i$ whose minimal elements are the fine mixed subdivisions.

\subsection{Basic sums and zonotopes} \label{sec:basiczono}

Let $A_1$, \dots, $A_n$ be as in the previous section. Let $\Delta^{m-1} = \{e_1,\dots,e_m\}$ be the standard basis of $\bb R^m$. We will say that $\sum A_i$ is \emph{basic} if $A_i \subseteq \Delta^{m-1}$ for all $i$.

The faces of a basic sum $\sum A_i$ can be described as follows: Given a weak ordering $\le_w$ of $\Delta^{m-1}$, for any $S \subseteq \Delta^{m-1}$ define $\min(\le_w, S)$ to be the set of minimal elements of $S$ with respect to $\le_w$. Then $\sum B_i$ is a face of $\sum A_i$ if and only if there is some weak ordering $\le_w$ such that for all $i$, either $B_i = \emptyset$ or $B_i = \min(\le_w, A_i)$. If $B_i = \min(\le_w, A_i)$ for all $i$, then we say $\sum B_i$ is \emph{induced} by $\le _w$.

Let $A_1$, \dots, $A_n$ be as in the previous section, and assume additionally that $\abs{A_i} = 2$ for all $i$. Then $\sum A_i$ is called a \emph{zonotope}, and its mixed subdivisions are called \emph{zonotopal tilings}. A fine zonotopal tiling is called a \emph{tight} zonotopal tiling.

\begin{exam}
For all $1 \le i < j \le m$, let $A_{ij} = \{e_i,e_j\}$. Then $\Pi^{m-1} := \sum_{1 \le i < j \le m} A_{ij}$ is a basic zonotope called the $(m-1)$-dimensional \emph{permutohedron}. Every weak ordering $\le_w$ of $\Delta^{m-1}$ induces a different face $\sum \min(\le_w, A_i)$ of the permutohedron.
\end{exam}

\subsection{Coherent mixed subdivisions}

Let $A_1$, \dots, $A_n$ be finite subsets of $\bb R^m$. For each $i = 1$, \dots, $n$, let $\omega_i : A_i \to \bb R$ be a function. For a mixed cell $\sum B_i$ of $\sum A_i$, define the \emph{lift} $(\sum B_i)^\omega$ by
\[
\left(\sum B_i \right)^\omega := \sum B_i^{\omega_i}.
\]
We call a mixed cell $\sum F_i$ of $(\sum A_i)^\omega$ a \emph{lower face} of $(\sum A_i)^\omega$ if there is a linear functional $\phi \in (\bb R^m \times \bb R)^\ast$ such that $\phi(0,1) > 0$ and for each $i$, either $F_i$ is empty or $F_i$ is the set of points which minimize $\phi$ on $A_i^{\omega_i}$. The collection of all $\sum B_i$ such that $(\sum B_i)^\omega$ is a lower face of $(\sum A_i)^\omega$ is a mixed subdivision of $\sum A_i$ called the \emph{coherent} mixed subdivision of $\sum A_i$ with respect to $\omega$. We denote it by $\ms S_{\sum A_i}^\omega$.

In the case where $\sum A_i$ is basic, we have the following characterization of $\ms S_{\sum A_i}^\omega$.

\begin{thm}[Develin and Sturmfels \cite{DS04}] \label{fans}
Let $\sum A_i$ be basic, and let $\omega_i : A_i \to \bb R$ be functions. For any $x = (x_1,\dotsc,x_m) \in \bb R^m$ and $A_i$, let $\type(x,\omega,A_i)$ be the set of all $e_j \in A_i$ such that
\[
x_j - \omega_i(e_j) = \max\{ x_k - \omega_i(e_k) : e_k \in A_i \}.
\]
Then $\sum B_i$ is an element of $\ms S_{\sum A_i}^\omega$ if and only if there is some $x \in \bb R^m$ such that for all $i$, either $B_i = \emptyset$ or $B_i = \type(x,\omega,A_i)$.
\end{thm}

\subsection{The Cayley trick}

Let $A_1$, \dots, $A_n$ be as before. Let $\Delta^{n-1} = \{f_1,\dots,f_n\}$ be the standard basis of $\bb R^n$. We define the \emph{Cayley embedding} of $\sum A_i$ to be the following set in $\bb R^m \times \bb R^n$:
\[
\mc C\left( \sum A_i \right) := \bigcup_{i=1}^n \{ (x,f_i) : x \in A_i \}.
\]
The \emph{Cayley trick} says the following.

\begin{thm}[Sturmfels \cite{Stu94}, Huber et al.\ \cite{HRS00}]
The following are true.
\begin{enumerate}
\item $\mc C$ is a bijection between the mixed cells of $\sum A_i$ and the cells of $\mc C(\sum A_i)$, and this map preserves facial relations.
\item For any mixed subdivision $\ms S$ of $\sum A_i$, the collection $\mc C(\ms S)$ is a subdivision of $\mc C(\sum A_i)$. The map $\ms S \mapsto \mc C(\ms S)$ is a poset isomorphism between the mixed subdivisions of $\sum A_i$ and the subdivisions of $\mc C(\sum A_i)$.
\item If $\omega_i : A_i \to \bb R$ are functions for $i = 1$, \dots, $n$ and $\mc C(\omega) : \mc C(\sum A_i) \to \bb R$ is defined as $\mc C(\omega)(x,f_i) = \omega_i(x)$, then
\[
\mc C \left(\ms S_{\sum A_i}^\omega \right) = \ms S_{\mc C(\sum A_i)}^{\mc C(\omega)}.
\]
\end{enumerate}
\end{thm}

\begin{exam}
If $\sum A_i$ is basic, then $\mc C(\sum A_i)$ is a a subset of $\Delta^{m-1} \times \Delta^{n-1}$.
If $A_i = \Delta^{m-1}$ for all $i$, then $\mc C(\sum A_i) = \Delta^{m-1} \times \Delta^{n-1}$. 
\end{exam}

We say that two fine mixed subdivisions $\ms T$ and $\ms T'$ are connected by a \emph{flip} if the triangulations $\mc C(\ms T)$ and $\mc C(\ms T')$ are connected by a flip.

\section{A zonotope with disconnected flip graph}

\subsection{Zonotopal tilings of the 3-permutohedron}

We are now ready to construct a zonotope and a nontrivial component of its flip graph. The construction uses the 3-dimensional permutohedron and eight of its zonotopal tilings.

It will be notationally easier to work with Cayley embeddings of zonotopes rather than zonotopes themselves. Thus we will identify a zonotope with its Cayley embedding.

For a set $S$, let $\Gamma_S^k$ denote the set of all ordered $k$-tuples $(i_1,\dotsc,i_k)$ of distinct $i_1$, \dots, $i_k \in S$ under the equivalence relation $(i_1,\dotsc,i_k) \sim (i_2,\dotsc,i_k,i_1)$. We will use $(i_1 \dotsb i_k)$ to denote the equivalence class of $(i_1,\dotsc,i_k)$ in $\Gamma_S^k$. We write $-(i_1 \dotsb i_k)$ to denote $(i_k \dotsb i_1)$. We abbreviate $\Gamma_{[n]}^k$ as $\Gamma_n^k$.

Let $\Delta^{m-1} = \{e_1,\dotsc,e_m\}$ be the standard basis for $\bb R^m$ and let $\Delta^{\binom{m}{2}-1} := \{ f_\alpha \}_{\alpha \in \Gamma_m^2}$ be the standard basis of $\bb R^{\binom{m}{2}}$. Then
\[
\Pi^{m-1} := \bigcup_{(ij) \in \Gamma_m^2} \{ (e_i,f_{(ij)}), (e_j,f_{(ij)}) \}
\]
is the (Cayley embedding) of the $(m-1)$-dimensional permutohedron.

For any $(ijk) \in \Gamma_m^3$, we have a circuit $X_{(ijk)} = (X_{(ijk)}^+,X_{(ijk)}^-)$  in $\Pi^{m-1}$ with affine dependence relation
\[
(e_i,f_{(ij)}) - (e_j,f_{(ij)}) + (e_j,f_{(jk)}) - (e_k,f_{(jk)}) + (e_k,f_{(ki)}) - (e_i,f_{(ki)})
\]
and $X_{(ijk)}^+$, $X_{(ijk)}^-$ defined in terms of this affine relation. Define $\ms T_{(ijk)} := \ms T_{X_{(ijk)}}^+$. Note that $\ms T_{-\gamma} = \ms T_{X_{\gamma}}^-$.

We now set $m=4$. We will construct eight different tilings of $\Pi^3$, each indexed by a different element of $\Gamma_4^3$. Fix some $\gamma = (ijk) \in \Gamma_4^3$. Let $\omega : \Pi^3 \to \bb R$ be the function with
\[
\omega(e_i,f_{(ij)}) = \omega(e_j,f_{(jk)}) = \omega(e_k,f_{(ki)}) = 1
\]
and $\omega(x) = 0$ for all other $x \in \Pi^3$. We define $\ms T_{\Pi^3}^\gamma := \ms S_{\Pi^3}^\omega$. It is easy to check (using Theorem~\ref{fans}, for example) that $\ms T_{\Pi^3}^\gamma$ is a triangulation.

For each $\gamma' \in \Gamma_4^3$, the circuit $X_{\gamma'}$ is a face of $\Pi^3$, and thus $\ms T_{\Pi^3}^\gamma$ contains the triangulation it induces on $X_{\gamma'}$. From the definition of $\omega$, we have
\[
\omega(e_i,f_{(ij)}) - \omega(e_j,f_{(ij)}) + \omega(e_j,f_{(jk)}) - \omega(e_k,f_{(jk)}) + \omega(e_k,f_{(ki)}) - \omega(e_i,f_{(ki)}) > 0
\]
and if $l$ is the element of $[4] \minus \{i,j,k\}$, we have
\begin{align*}
\omega(e_i,f_{(ij)}) - \omega(e_j,f_{(ij)}) + \omega(e_j,f_{(jl)}) - \omega(e_l,f_{(jl)}) + \omega(e_l,f_{(li)}) - \omega(e_i,f_{(li)}) &> 0 \\
\omega(e_j,f_{(jk)}) - \omega(e_k,f_{(jk)}) + \omega(e_k,f_{(kl)}) - \omega(e_l,f_{(kl)}) + \omega(e_l,f_{(lj)}) - \omega(e_j,f_{(lj)}) &> 0 \\
\omega(e_k,f_{(ki)}) - \omega(e_i,f_{(ki)}) + \omega(e_i,f_{(il)}) - \omega(e_l,f_{(il)}) + \omega(e_l,f_{(lk)}) - \omega(e_k,f_{(lk)}) &> 0.
\end{align*}
Thus, by Proposition~\ref{circuittiling}, $\ms T_{\Pi_3}^\gamma$ induces the following triangulations on the circuits $X_{\gamma'}$:
\begin{equation} \label{threeflips}
\ms T_{(ijk)}, \ms T_{(ijl)}, \ms T_{(jkl)}, \ms T_{(kil)} \subseteq \ms T_{\Pi^3}^\gamma.
\end{equation}

\subsection{A group action on $\Gamma_4^3$}

The key property of $\ms T_{\Pi^3}^\gamma$ is that it only has flips on the circuits $X_{(ijl)}$, $X_{(jkl)}$, and $X_{(kil)}$. The idea will be to tile a larger zonotope with 3-permutohedra and then tile each 3-permutohedron with some $\ms T_{\Pi^3}^\gamma$ so that in the end, no circuit of size six can be flipped. To help with this construction, we will define a group action on $\Gamma_4^3$.

For each $\gamma  = (ijk) \in \Gamma_4^3$, we define a function $o_\gamma : \binom{[4]}{3} \to \Gamma_4^3$ by
\begin{align*}
o_\gamma(\{i,j,k\}) &= (ijk) \\
o_\gamma(\{i,j,l\}) &= (ijl) \\
o_\gamma(\{j,k,l\}) &= (jkl) \\
o_\gamma(\{k,i,l\}) &= (kil)
\end{align*}
where $\{l\} = [4] \minus \{i,j,k\}$. Hence, equation \eqref{threeflips} implies $\ms T_{o_\gamma(S)} \subseteq \ms T_{\Pi^3}^\gamma$ for all $S \in \binom{[4]}{3}$. It is easy to check that $\gamma$ is determined by $o_\gamma$.

Now, we will map each $\alpha \in \Gamma_4^2$ to a permutation $\pi_\alpha : \Gamma_4^3 \to \Gamma_4^3$. This map is completely determined by the following rules: For any distinct $i$, $j$, $k$, $l \in [4]$, we have
\begin{align*}
\pi_{(ij)}(ijk) &= (jil) \\
\pi_{(kl)}(ijk) &= (ijl).
\end{align*}
Let $G_{\Gamma_4^3}$ be the permutation group of $\Gamma_4^3$ generated by all the $\pi_\alpha$.

\begin{prop} \label{groupaction}
The following are true.
\begin{enumerate}
\item Every element of $G_{\Gamma_4^3}$ is an involution, and $G_{\Gamma_4^3}$ is abelian and transitive on $\Gamma_4^3$.
\item For $l \in [4]$, let $H_l$ be the subgroup of $G_{\Gamma_4^3}$ generated by $\pi_{(il)}$ for all $i \in [4] \minus \{l\}$. Let $i$,$j$,$k \in [4] \minus \{l\}$ be distinct, and let $\Gamma_4^3(ijk)$ be the set of all $\gamma \in \Gamma_4^3$ such that $o_\gamma(\{i,j,k\}) = (ijk)$. Then $\Gamma_4^3(ijk)$ is an orbit of $H_l$.
\end{enumerate}
\end{prop}

\begin{proof}
Since each $\gamma$ is determined by $o_\gamma$, we can view $G_{\Gamma_4^3}$ as an action on the set of functions $o_\gamma$.
We check that for all distinct $i$, $j$, $k$, $l \in [4]$ and $\gamma \in \Gamma_4^3$, we have
\begin{align*}
o_{\pi_{(ij)}\gamma}(\{i,j,k\}) &= -o_\gamma(\{i,j,k\}) \\
o_{\pi_{(ij)}\gamma}(\{i,j,l\}) &= -o_\gamma(\{i,j,l\}) \\
o_{\pi_{(ij)}\gamma}(\{j,k,l\}) &= o_\gamma(\{j,k,l\}) \\
o_{\pi_{(ij)}\gamma}(\{k,i,l\}) &= o_\gamma(\{k,i,l\})
\end{align*}
It follows that we can embed $G_{\Gamma_4^3}$ as a subgroup of $\bb Z_2^4$. This implies that every element of $G_{\Gamma_4^3}$ is an involution and $G_{\Gamma_4^3}$ is abelian. It is also easy to check from the above action on the $o_\gamma$ that every element of $\Gamma_4^3$ has orbit of size 8, and hence $G_{\Gamma_4^3}$ is transitive.

From the above action on $o_\gamma$, we see that $H_l$ maps $\Gamma_4^3(ijk)$ to itself and every element of $\Gamma_4^3(ijk)$ has orbit of size 4 under $H_l$. Since $\abs{\Gamma_4^3(ijk)} = 4$, $\Gamma_4^3(ijk)$ is an orbit of $H_l$.
\end{proof}

\subsection{A zonotope and a component of its flip graph} \label{zonocomponent}

Let $N$ be a positive integer to be determined later. For each distinct $i$, $j \in [4]$ and $-N \le r \le N$, we create a variable $f_{ij}^r$, and we make the identification of variables
\[
f_{ij}^r = f_{ji}^{-r}.
\]
Let $\{f_{ij}^r\}_{1 \le i < j \le 4, -N \le r \le N}$ be the standard basis for $\bb R^{6(2N+1)}$. Let
\begin{equation} \label{P3}
\Pi := \bigcup_{1 \le i < j \le 4} \bigcup_{-N \le r \le N} \{(e_i,f_{ij}^r), (e_j,f_{ij}^r) \}.
\end{equation}
be the 3-dimensional permutohedron with $2N+1$ copies of each generating vector. We will now prove our first main result.

\begin{thm} \label{main1}
For large enough $N$, the flip graph of $\Pi$ is not connected.
\end{thm}

For distinct $i$, $j$, $k \in [4]$ and for any $-N \le r, s, t \le N$, let $X_{ijk}^{rst} = ((X_{ijk}^{rst})^+, (X_{ijk}^{rst})^-)$ be the circuit with affine dependence relation
\[
(e_i,f_{ij}^r) - (e_j,f_{ij}^r) + (e_j,f_{jk}^s) - (e_k,f_{jk}^s) + (e_k,f_{ki}^t) - (e_i,f_{ki}^t)
\]
and $(X_{ijk}^{rst})^+$, $(X_{ijk}^{rst})^-$ defined in terms of this relation. Let $\ms T_{ijk}^{rst} := \ms T_{X_{ijk}^{rst}}^+$. The following Lemma identifies a component of the flip graph of $\Pi$.

\begin{lem} \label{ensemble}
Let $\ms C$ be a collection of triangulations of the form $\ms T_{ijk}^{rst}$ such that for all $\ms T_{ijk}^{rst} \in \ms C$ and $\{l\} = [4] \minus \{i,j,k\}$, there exist $1 \le u,v,w \le N$ such that
\[
\ms T_{ijl}^{rv(-u)}, \ms T_{jkl}^{sw(-v)}, \ms T_{kil}^{tu(-w)} \in \ms C.
\]
Let $\mc S_{\ms C}$ be the set of all triangulations of $\Pi$ which contain every element of $\ms C$ as a subset. Then $\mc S_{\ms C}$ is closed under flips.
\end{lem}

\begin{proof}
The proof follows immediately from the following two facts, both of which are proved in the appendix.

\begin{prop} \label{zonoblock}
Let $\ms T$ be a triangulation of $\Pi$, and suppose that there are distinct $i$, $j$, $k$, $l \in [4]$ and $1 \le r,s,t,u,v,w \le N$ such that
\[
\ms T_{ijl}^{rv(-u)}, \ms T_{jkl}^{sw(-v)}, \ms T_{kil}^{tu(-w)} \subseteq \ms T.
\]
Then $\ms T$ does not have a flip supported on $X_{ijk}^{rst}$.
\end{prop}

\begin{prop} \label{zonoshift}
Let $\ms T$ be a triangulation of $\Pi$  such that $\ms T_{ijk}^{rst} \subseteq \ms T$. Let $\ms T'$ be the result of a flip on $\ms T$ which is not supported on $X_{ijk}^{rst}$. Then $\ms T_{ijk}^{rst} \subseteq \ms T'$.
\end{prop}
\end{proof}

It remains to prove that there is some $\ms C$ for which $\mc S_{\ms C}$ is neither empty nor the whole set of triangulations of $\Pi$. This will be done in the next section.

\subsection{Construction of $\ms C$ and some $\ms T \in \mc S_{\ms C}$} \label{zonoconstruction}

We first consider the regular subdivision $\ms S_\Pi^\omega$ where $\omega : \Pi \to \bb R$ is a function such that
\[
\omega(e_i, f_{ij}^r) - \omega(e_j, f_{ij}^r) = r
\]
for all distinct $i$, $j \in [4]$ and $-N \le r \le N$. The cells of $\ms S_\Pi^\omega$ are described as follows.

\begin{prop} \label{celllabeling}
Let $\ms X$ be the set of $x = (x_1,x_2,x_3,x_4) \in \bb R^4$ such that $x_1 + \dotsb + x_4 = 0$ and if $ijkl$ is a permutation of $[4]$ such that $x_i \ge x_j \ge x_k \ge x_l$, then $x_i - x_j$, $x_j - x_k$, and $x_k - x_l$ are integers at most $N$. Let $\ms X^\ast$ be the set of $x \in \ms X$ such that $\abs{x_i-x_j} \le N$ for all $i$, $j \in [n]$. The following are true.
\begin{enumerate}
\item The map $C(x) = \{ (e_i, f_{ij}^r) : x_i - x_j \ge r \}$ is a bijection from $\ms X$ to the maximal cells of $\ms S_\Pi^\omega$.
\item If $x \in \ms X^\ast$, then $C(x)$ is the Cayley embedding of a translated 3-permutohedron. Specifically, $C(x) = \Pi(x) \cup D$, where
\[
\Pi(x) := \bigcup_{1 \le i < j \le 4} \{ (e_i, f_{ij}^{x_i - x_j}), (e_j, f_{ij}^{x_i-x_j}) \}
\]
and $D$ is a simplex affinely independent to $\Pi(x)$.
\item Let $x \in \ms X \minus \ms X^\ast$. Suppose $F_1$, \dots, $F_k$ are faces of $C(x)$ and $\ms T_1$, \dots, $\ms T_k$ are triangulations of these faces, respectively, which agree on intersections of these faces. Then there is a triangulation of $C(x)$ which contains $\ms T_1$, \dots, $\ms T_k$ as subsets.
\end{enumerate}
\end{prop}

\begin{proof}
Part 1 follows from Theorem~\ref{fans}. Part 2 immediately follows. The only nontrivial case of part 3 is when $x$ satisfies $x_i \ge x_j \ge x_k \ge x_l$ for some permutation $ijkl$ of $[4]$, $x_i - x_k \le N$, $x_j - x_l \le N$, and $x_i - x_l > N$. In this case $C(x)$ is of the form $X_{ijk}^{rst} \cup X_{jkl}^{suv} \cup D$, where $D$ is a simplex affinely independent to $X_{ijk}^{rst} \cup X_{jkl}^{suv}$. By Proposition~\ref{glue} (or by an easy check), any triangulations of $X_{ijk}^{rst}$ and $X_{jkl}^{suv}$ can be extended to a triangulation of $X_{ijk}^{rst} \cup X_{jkl}^{suv}$, and hence to a triangulation of $C(x)$.
\end{proof}

For each $x \in \ms X^\ast$, we have an affine isomorphism $\Pi^3 \to \Pi(x)$ given by $f_{(ij)} \mapsto f_{ij}^{x_i-x_j}$. For each $\gamma \in \Gamma_4^3$, let $\ms T_{\Pi(x)}^\gamma$ be the image of $\ms T_{\Pi^3}^\gamma$ under this isomorphism.

We will now choose a random triangulation of every $C(x)$, $x \in \ms X^\ast$, as follows:
\begin{enumerate}
\item For each $1 \le i < j \le 4$ and $-N \le r \le N$, let $g_{ij}^r = g_{ji}^{-r}$ be an independent random element of $G_{\Gamma_4^3}$ which is 1 with probability $1/2$ and $\pi_{(ij)}$ with probability 1/2.
\item For each $x \in \ms X^\ast$, triangulate $\Pi(x)$ by $\ms T_{\Pi(x)}^{\gamma(x)}$, where
\[
\gamma(x) := \left( \prod_{1 \le i < j \le 4} g_{ij}^{x_i-x_j} \right) (123).
\]
\item Extend $\ms T_{\Pi(x)}^{\gamma(x)}$ uniquely to a triangulation $\ms T_{C(x)}$ of $C(x)$.
\end{enumerate}

\begin{prop} \label{permsagree}
For any two $x$, $x' \in \ms X^\ast$, the triangulations $\ms T_{C(x)}$ and $\ms T_{C(x')}$ agree on the common face of $C(x)$ and $C(x')$.
\end{prop}

\begin{proof}
The only non-trivial case is when $C(x) \cap C(x')$ contains a circuit $X_{ijk}^{rst}$. We need to check that $\ms T_{C(x)}$ and $\ms T_{C(x')}$ agree on this circuit. If $X_{ijk}^{rst} \subseteq C(x) \cap C(x')$, then
\begin{align*}
x_i - x_j = x'_i - x'_j &= r \\
x_j - x_k = x'_j - x'_k &= s \\
x_k - x_i = x'_k = x'_i &= t.
\end{align*}
On the other hand, by Proposition~\ref{groupaction}(2), $o_{\gamma(x)}(\{i,j,k\})$ depends only on $g_{ij}^{x_i-x_j}$, $g_{jk}^{x_j-x_k}$, and $g_{ki}^{x_k-x_i}$. It follows that $o_{\gamma(x)}(\{i,j,k\}) = o_{\gamma(x')}(\{i,j,k\})$. Thus $\ms T_{\Pi(x)}^{\gamma(x)}$ and $\ms T_{\Pi(x')}^{\gamma(x')}$ contain the same triangulation of $X_{ijk}^{rst}$, as desired.
\end{proof}

By Proposition~\ref{permsagree} and Proposition~\ref{celllabeling}(3), we can thus extend the above triangulations of the $C(x)$ to a triangulation $\ms T$ of $\Pi$. Let $\ms C$ be the collection of all triangulations $\ms T_{ijk}^{rst} \subseteq \ms T$ with $i$, $j$, $k \in [4]$ distinct, $-N \le r,s,t \le N$, and $r+s+t=0$. We prove that $\ms C$ satisfies the hypotheses of Lemma~\ref{ensemble}. We will actually prove the following stronger statement, which we will need in the next section.

\begin{prop} \label{probably}
For large enough $N$, with probability greater than 0, $\ms T$ and $\ms C$ satisfy the following:
\begin{enumerate} \renewcommand{\labelenumi}{(\Alph{enumi})}
\item For every distinct $i$, $j$, $k \in [4]$ and $-N \le r \le N$, there exist $-N \le s,t \le N$ such that $\ms T_{ijk}^{rst} \in \ms C$.
\item For every $\ms T_{ijk}^{rst} \in \ms C$ and $\gamma \in \Gamma_4^3(ijk)$, there exists $x \in \ms X^\ast$ such that $X_{ijk}^{rst} \subseteq \Pi(x)$ and $\ms T[\Pi(x)] = \ms T_{\Pi(x)}^\gamma$. In particular, when $\gamma = (ijk)$, this implies there is some $-N \le u,v,w \le N$ such that
\[
\ms T_{ijl}^{rv(-u)}, \ms T_{jkl}^{sw(-v)}, \ms T_{kil}^{tu(-w)} \in \ms C
\]
where $\{l\} = [4] \minus \{i,j,k\}$.
\end{enumerate}
\end{prop}

\begin{proof}
First, note that for any distinct $i$, $j$, $k \in [4]$ and $-N \le r,s,t \le N$ with $r+s+t=0$, there is some $x \in \ms X^\ast$ such that $X_{ijk}^{rst} \subseteq \Pi(x)$. Hence $\ms C$ contains a triangulation of $X_{ijk}^{rst}$. Letting $\gamma = g_{ij}^r g_{jk}^s g_{ki}^t (123)$, we have
\begin{equation} \label{whichcircuit}
\ms T_{ijk}^{rst} \in \ms C \quad\text{if and only if } o_\gamma(\{i,j,k\}) = (ijk).
\end{equation}

We first bound the probability that (A) does not hold. Fix distinct $i$, $j$, $k \in [4]$ and $-N \le r \le N$. Let $H$ be the set of all ordered pairs $(s,t) \in [-N,N]^2$ with $r+s+t=0$. Note that $\abs{H} \ge N$. For each $(s,t) \in H$ and $\gamma(s,t) := g_{ij}^r g_{jk}^s g_{ki}^t (123)$, it is easy to see from the proof of Proposition~\ref{groupaction} that $o_{\gamma(s,t)}(\{i,j,k\}) = (ijk)$ with probability $1/2$. In fact, this probability does not change if we fix $g_{ij}^r$, so for all $(s,t) \in H$ these probabilities are mutually independent. Now, from \eqref{whichcircuit}, there does not exist $(s,t) \in H$ such that $\ms T_{ijk}^{rst} \in \ms C$ if and only if $o_{\gamma(s,t)}(\{i,j,k\}) \neq (ijk)$ for all $(s,t) \in H$. The probability this happens is
\[
\left( \frac{1}{2} \right)^{\abs{H}} \le \left( \frac{1}{2} \right)^{N}.
\]
By the union bound, the probability that this happens for some distinct $i$, $j$, $k \in [4]$ and $-N \le r \le N$ is at most
\begin{equation}
24(2N+1)\left( \frac{1}{2} \right)^{N}.
\end{equation}
This gives an upper bound on the probability of (A) not happening.

We now do the same for (B). Fix $\ms T_{ijk}^{rst} \in \ms C$ and $\gamma \in \Gamma_4^3(ijk)$. Let $H$ be the set of all $x \in \ms X^\ast$ such that $X_{ijk}^{rst} \subseteq \Pi(x)$. Note that $\abs{H} \ge N$. Let $\{l\} = [4] \minus \{i,j,k\}$.

We have $\ms T_{ijk}^{rst} \in \ms C$, which happens if and only if $g_{ij}^r g_{jk}^s g_{ki}^t (123) \in \Gamma_4^3(ijk)$. Suppose we fix $g_{ij}^r$, $g_{jk}^s$, and $g_{ki}^t$ such that $g_{ij}^r g_{jk}^s g_{ki}^t (123) \in \Gamma_4^3(ijk)$. Then for each $x \in H$, it follows from Proposition~\ref{groupaction} and the definition of $\gamma(x)$ that $\gamma(x) = \gamma$ with probability $1/4$. Moreover, since $(x_i - x_l, x_j - x_l, x_k - x_l)$ is different for each $x \in H$, these probabilities are mutually independent for all $x \in H$. Now, for each $x \in \ms X^\ast$ we have $\ms T[\Pi(x)] = \ms T_{\Pi(x)}^{\gamma(x)}$. Thus the probability that there is no $x \in H$ with $\ms T[\Pi(x)] = \ms T_{\Pi(x)}^\gamma$ is
\[
\left( \frac{3}{4} \right)^{\abs{H}} \le \left( \frac{3}{4} \right)^{N}.
\]
The probability that this occurs for some distinct $i$, $j$, $k \in [4]$, $-N \le r,s,t \le N$ with $r+s+t=0$, and $\gamma \in \Gamma_4^3(ijk)$ is thus at most
\begin{equation}
96(2N+1)^2\left( \frac{3}{4} \right)^{N}.
\end{equation}
Hence, the probability that either (A) or (B) does not hold is at most
\[
24(2N+1)\left( \frac{1}{2} \right)^{N} + 96(2N+1)^2\left( \frac{3}{4} \right)^{N}
\]
which for large enough $N$ (specifically, $N \ge 48$) is less than 1.
\end{proof}

Thus there exists $\ms C$ which satisfies the hypotheses of Lemma~\ref{ensemble} and $\ms T \in \mc S_{\ms C}$. There are triangulations of $\Pi$ which are not in $\mc S_{\ms C}$; for example, a different choice of the $g_{ij}^r$ would yield a triangulation which does not contain every element of $\ms C$. This proves Theorem~\ref{main1}.

\section{A product of two simplices with disconnected flip graph}

We will use the construction from the previous section to show that the product of two simplices does not in general have connected flip graph. The idea will be to go up one dimension and construct multiple copies of $\ms T$ in different directions in this space.

\subsection{A component of the flip graph} \label{sec:tricomponent}

For each $\alpha \in \Gamma_5^2$, construct a finite set $\Delta_\alpha$ of variables. We will determine the size of this set later. Let $\Delta^4 := \{e_1,\dots,e_5\}$ be the standard basis for $\bb R^5$, and let $\Delta^{n-1} := \bigcup_{\alpha \in \Gamma_5^2} \Delta_\alpha$ be the standard basis for $\bb R^n$, where $n = \sum_{\alpha \in \Gamma_5^2} \abs{\Delta_\alpha}$. Let $A := \Delta^4 \times \Delta^{n-1}$. We prove the following.

\begin{thm} \label{main2}
For large enough $n$, the flip graph of $A$ is not connected.
\end{thm}

For distinct $i_1$, \dots, $i_t \in [5]$ and distinct $f_1$, \dots, $f_t \in \Delta^{n-1}$, let
\[
X_{i_1 \dotsb i_t}^{f_1 \dotsb f_t} = ((X_{i_1 \dotsb i_t}^{f_1 \dotsb f_t})^+,(X_{i_1 \dotsb i_t}^{f_1 \dotsb f_t})^-)
\]
be the circuit in $A$ with affine dependence relation
\[
(e_{i_1}, f_1) - (e_{i_2}, f_1) + (e_{i_2}, f_2) - (e_{i_3}, f_3) + \dotsb + (e_{i_t}, f_t) - (e_{i_1}, f_t)
\]
and $(X_{i_1 \dotsb i_t}^{f_1 \dotsb f_t})^+$, $(X_{i_1 \dotsb i_t}^{f_1 \dotsb f_t})^-$ defined in terms of this relation. Let
\[
\ms T_{i_1 \dotsb i_t}^{f_1 \dotsb f_t} := \ms T_{X_{i_1 \dotsb i_t}^{f_1 \dotsb f_t}}^+.
\]
If $f_1 \in \Delta_{(i_1i_2)}$, $f_2 \in \Delta_{(i_2i_3)}$, \dots, $f_t \in \Delta_{(i_ti_1)}$, then we call the circuit $X_{i_1 \dotsb i_t}^{f_1 \dotsb f_t}$ and the triangulation $\ms T_{i_1 \dotsb i_t}^{f_1 \dotsb f_t}$ \emph{zonotopal}.

We now identify a component of the flip graph of $A$.

\begin{lem} \label{ensemble2}
For each $S \in \binom{[5]}{4}$ and distinct $i$, $j \in S$, let $\ms C_{S,i,j}$ be a collection of zonotopal triangulations of the form $\ms T_{i_1i_2i_3}^{f_1f_2f_3}$ where $i_1$, $i_2$, $i_3 \in S$. Assume that for each $S \in \binom{[5]}{4}$ and distinct $i$, $j \in S$, and for each $\ms T_{i_1i_2i_3}^{f_1 f_2 f_3} \in \ms C_{S,i,j}$, the following are true.
\begin{enumerate}
\item If $\{i_4\} = S \minus \{i_1,i_2,i_3\}$, then there exist $f_1' \in \Delta_{(i_1i_4)}$, $f_2' \in \Delta_{(i_2i_4)}$, and $f_3' \in \Delta_{(i_3i_4)}$ such that
\[
\ms T_{i_1i_2i_4}^{f_1f_2'f_1'}, \ms T_{i_2i_3i_4}^{f_2f_3'f_2'}, \ms T_{i_3i_1i_4}^{f_3f_1'f_3'} \in \ms C_{S,i,j}.
\]
\item If $i_1 = i$, $i_2 = j$, and $i_4$ is as above, then there exist $f_1' \in \Delta_{(i_1i_4)}$, $f_2' \in \Delta_{(i_2i_4)}$, and $f_3' \in \Delta_{(i_3i_4)}$ such that
\[
\ms T_{i_1i_2i_4}^{f_1f_2'f_1'}, \ms T_{i_3i_2i_4}^{f_2f_2'f_3'}, \ms T_{i_1i_3i_4}^{f_3f_3'f_1'} \in \ms C_{S,i,j}.
\]
\end{enumerate}
In addition, assume the following about the $\ms C_{S,i,j}$.
\begin{enumerate} \setcounter{enumi}{2}
\item\label{e2spans} For each distinct $i$, $j$, $l \in [5]$ and $f_1 \in \Delta_{(ij)}$, there exist $k \in [5] \minus \{i,j,l\}$, $f_2 \in \Delta_{(jk)}$, and $f_3 \in \Delta_{(ki)}$ such that
\[
\ms T_{ijk}^{f_1f_2f_3} \in \ms C_{[5] \minus \{l\}, i, j}.
\]
\end{enumerate}
Let $\mc S_{\ms C}$ be the set of all triangulations $\ms T$ of $A$ satisfying the following.
\begin{enumerate}
\renewcommand{\theenumi}{(\roman{enumi})}
\renewcommand{\labelenumi}{\theenumi}
\item\label{e2contains} For each $S \in \binom{[5]}{4}$, distinct $i$, $j \in S$, and $\ms T_{i_1i_2i_3}^{f_1 f_2 f_3} \in \ms C_{S,i,j}$, we have
\[
\ms T_{i_1i_2i_3}^{f_1 f_2 f_3} \subseteq \ms T.
\]
\item\label{e2Contains} For each $S \in \binom{[5]}{4}$, distinct $i$, $j \in S$, and $\ms T_{ijk}^{f_1 f_2 f_3} \in \ms C_{S,i,j}$ where $k \in S \minus \{i,j\}$, if $\{l\} = [5] \minus S$, then for any $f \in \Delta_{(il)}$ we have
\[
X_{ijk}^{f_1 f_2 f_3} \minus\{(e_i,f_1)\} \cup \{(e_i,f)\} \in \ms T.
\]
\end{enumerate}
Then $\mc S_{\ms C}$ is closed under flips.
\end{lem}

\begin{proof}
We need the following three general facts about triangulations of $A$. The first is true for all point sets $A$ and is proved in \cite{Liu16}. The other two are proved in the Appendix.

\begin{prop} \label{flip}
Let $\ms T$ be a triangulation of $A$ and let $X = (X^+,X^-)$ be a circuit in $A$. Suppose that $X^- \in \ms T$. Then $\ms T$ has a flip supported on $(X^+,X^-)$ if and only if there is no maximal simplex $\tau \in \ms T$ with $X^- \subseteq \tau$ and $\abs{X \cap \tau} \le \abs{X} - 2$.
\end{prop}

\begin{prop} \label{triblock}
Let $\ms T$ be a triangulation of $A$. Let $i$, $j$, $k$, $l \in [5]$ be distinct, and for each $\alpha \in \Gamma_{\{i,j,k,l\}}^2$, let $f_\alpha \in \Delta^{n-1}$. Suppose that
\[
\ms T_{jkl}^{f_{(jk)}f_{(kl)}f_{(lj)}}, \ms T_{kil}^{f_{(ki)}f_{(il)}f_{(lk)}} \subseteq \ms T
\]
and
\[
X_{ijl}^{f_{(ij)}f_{(jl)}f_{(li)}} \minus \{ (e_i, f_{(ij)}) \} \in \ms T.
\]
Then $\ms T$ does not have a flip supported on $X_{ijk}^{f_{(ij)}f_{(jk)}f_{(ki)}}$.
\end{prop}

\begin{prop} \label{trishift}
Let $\ms T$ be a triangulation of $A$ and let $\ms T'$ be the result of a flip on $\ms T$ supported on $X = (X^+,X^-)$. Suppose that $\sigma \in \ms T$ and $\sigma \notin \ms T'$, and $G(\sigma)$ is connected. Then $\sigma$ contains a maximal simplex of $\ms T_X^+$.
\end{prop}

We also need the following two facts about flips on elements of $\mc S_{\ms C}$.

\begin{prop} \label{flip2}
Let $i$, $j$, $l \in [5]$ be distinct and let $f_1 \in \Delta_{(ij)}$, $f_2 \in \Delta_{(il)}$. Let $\ms T \in \mc S_\ms C$. Then $\ms T$ does not have a flip supported on $X_{ij}^{f_1f_2}$.
\end{prop}

\begin{proof}
By Property~\ref{e2spans} of $\ms C_{S,i,j}$, there exist $k \in [5] \minus \{i,j,l\}$, $f_2' \in \Delta_{(jk)}$, and $f_3' \in \Delta_{(ki)}$ such that $\ms T_{ijk}^{f_1f_2'f_3'} \in \ms C_{[5] \minus \{l\},i,j}$. By property~\ref{e2Contains} of $\mc S_{\ms C}$, it follows that
\[
\sigma := X_{ijk}^{f_1f_2'f_3'} \minus \{(e_i,f_1)\} \cup \{(e_i,f_2)\} \in \ms T.
\]
If $\tau$ is a maximal simplex in $\ms T$ containing $\sigma$, then $(X_{ij}^{f_1f_2})^- \subseteq \tau$ but $\abs{X_{ij}^{f_1f_2} \cap \tau} = 2$. So by Proposition~\ref{flip}, $\ms T$ does not have a flip supported on $X_{ij}^{f_1f_2}$.
\end{proof}

\begin{prop} \label{flip3}
Let $\ms T_{i_1i_2i_3}^{f_1 f_2 f_3} \in \ms C_{S,i,j}$. Let $\ms T \in \mc S_\ms C$. Then $\ms T$ does not have a flip supported on $X_{i_1i_2i_3}^{f_1 f_2 f_3}$.
\end{prop}

\begin{proof}
This is a direct corollary of Property 1 of $\ms C_{S,i,j}$ and Proposition~\ref{triblock}.
\end{proof}

We now proceed with the proof. Suppose that $\ms T \in \mc S_\ms C$, and let $\ms T'$ be the result of a flip on $\ms T$ supported on $X = (X^+,X^-)$.
We prove that properties (i) and (ii) hold for $\ms T'$.

\paragraph{Property~\ref{e2contains}}
Suppose that $\ms T_{i_1i_2i_3}^{f_1 f_2 f_3} \in \ms C_{S,i,j}$ and $\ms T_{i_1i_2i_3}^{f_1 f_2 f_3} \not\subseteq \ms T'$. Without loss of generality, assume that $X_{i_1i_2i_3}^{f_1 f_2 f_3} \minus \{(e_{i_1},f_1)\} \notin \ms T'$. By Proposition~\ref{trishift}, $X_{i_1i_2i_3}^{f_1 f_2 f_3} \minus \{(e_{i_1},f_1)\}$ contains a maximal simplex of $\ms T_X^+$. This leaves two cases for $X$.
\begin{enumerate}

\item Case 1: $X$ has size 4. Thus we can write $X = X_{i_1'i_2'}^{f_1'f_2'}$ for some $i_1'$, $i_2' \in \{i_1,i_2,i_3\}$, $f_1' \in \Delta_{(i_1'i_2')}$, and $f_2' \in \Delta_{(i_1'i_3')}$, where $\{i_3'\} = \{i_1,i_2,i_3\} \minus \{i_1',i_2'\}$. However, this contradicts Proposition~\ref{flip2}. So we cannot have $\abs{X} = 4$.

\item Case 2: $X = X_{i_1i_2i_3}^{f_1 f_2 f_3}$. This contradicts Proposition~\ref{flip3}.

\end{enumerate}
Hence we must have $\ms T_{i_1i_2i_3}^{f_1 f_2 f_3} \subseteq \ms T'$, as desired.

\paragraph{Property~\ref{e2Contains}}
Suppose that $\sigma := X_{ijk}^{f_1 f_2 f_3} \minus\{(e_i,f_1)\} \cup \{(e_i,f)\} \notin \ms T'$, with variables as defined in \ref{e2Contains}. By Proposition~\ref{trishift}, $\sigma$ contains a maximal simplex of $X$. By the same argument as in part~\ref{e2contains}, we cannot have $\abs{X} = 4$ or $X = X_{ijk}^{f_1 f_2 f_3}$. This leaves
\[
X = X_{jik}^{ff_3f_2}
\]
as the only possibility.

We show that $\ms T$ cannot have a flip on this circuit. Let $l' = S \minus \{i,j,k\}$. By Property~2 of $\ms C_{S,i,j}$, there exist $f_1' \in \Delta_{(il')}$, $f_2' \in \Delta_{(jl')}$, $f_3' \in \Delta_{(kl')}$ such that
\begin{equation} \label{shiftedblock}
\ms T_{ijl'}^{f_1f_2'f_1'}, \ms T_{kjl'}^{f_2f_2'f_3'}, \ms T_{ikl'}^{f_3f_3'f_1'} \in \ms C_{S,i,j}.
\end{equation}
By Property \ref{e2Contains}, the first of these inclusions implies that
\[
X_{ijl'}^{f_1f_2'f_1'} \minus \{(e_i,f_1)\} \cup \{(e_i,f)\} \in \ms T
\]
and hence
\[
X_{jil'}^{ff_1'f_2'} \minus \{(e_j,f)\} \in \ms T.
\]
This inclusion along with the last two inclusions of \eqref{shiftedblock} imply, by Proposition~\ref{triblock}, that $\ms T$ does not have a flip on $X_{jik}^{ff_3f_2}$, as desired.
\end{proof}

\subsection{Reduction to zonotopes}

Let
\[
\Pi := \bigcup_{(ij) \in \Gamma_5^2} \bigcup_{f \in \Delta_{(ij)}} \{ (e_i, f), (e_j,f) \}
\]
be the 4-permutohedron embedded in $A$. Suppose we have collections $\ms C_{S,i,j}$ which satisfy the conditions of Lemma~\ref{ensemble2}. Let $\ms T$ be a triangulation of $\Pi$. Notice that if properties \ref{e2contains} and \ref{e2Contains} of Lemma~\ref{ensemble2} hold for $\ms T$, then they hold for any collection containing $\ms T$. In particular, if $\ms T$ can be extended to a triangulation $\ms T'$ of $A$, then we will have some $\ms T' \in \mc S_{\ms C}$. The next Proposition guarantees we can always do this.

\begin{prop} \label{extensiontoA}
If $\ms T$ is a triangulation of $\Pi$, then there is a triangulation $\ms T'$ of $A$ with $\ms T \subseteq \ms T'$.
\end{prop}

\begin{proof}
Let $\ms S_A := \ms S_A^\omega$ be the regular subdivision of $A$ with $\omega : A \to \bb R$ defined as follows. For each distinct $i$, $j$, $k \in [5]$ and $f \in \Delta_{(ij)}$, let $\epsilon_{f,k} > 0$ be a generic positive real number, and define
\[
\omega(e_i,f) = 0 \qquad \omega(e_j,f) = 0 \qquad \omega(e_k,f) = \epsilon_{f,k}.
\]
Then $\ms S_A$ contains $\Pi$ as a cell, corresponding to $x = 0$ under the notation of Theorem~\ref{fans}. Since the $\epsilon_{f,k}$ are generic, all other cells of $\ms S_A$ can be written as $F \cup D$, where $F$ is a face of $\Pi$ and $D$ is a simplex affinely independent to $F$. Hence, any triangulation of the cell $\Pi$ can be extended to a refinement of $\ms S_A$ which is a triangulation of $A$.
\end{proof}

\subsection{Construction of a zonotopal tiling}

We now construct a triangulation of $\Pi$ from which we will obtain our collections $\ms C_{S,i,j}$. For each $(ij) \in \Gamma_5^2$, partition $\Delta_{(ij)}$ into the following sets: $\Delta_{i,j}$, $\Delta_{j,i}$, and
\[
\Delta_{S,i',j',(ij)} \text{ for each } S \in \binom{[5]}{4} \text{ with } i,j, \in S \text{ and distinct } i', j' \in S \text{ with } \{i',j'\} \neq \{i,j\}.
\]
Next, for each distinct $i$, $j \in [5]$, choose an element $f_{i,j}^\ast \in \Delta_{i,j}$. For each $S \in \binom{[5]}{4}$ with $i,j \in S$, let $\Delta_{S,i,j,(ij)} \subseteq \Delta_{i,j}$ be sets such that
\[
\Delta_{i,j} = \bigcup_{\substack{S \in \binom{[5]}{4} \\ i,j \in S}} \Delta_{S,i,j,(ij)}
\]
and
\[
\Delta_{S,i,j,(ij)} \cap \Delta_{S',i,j,(ij)} = \{f_{i,j}^\ast\} \text{ for each distinct } S, S' \in \binom{[5]}{4}, i,j \in S, S'.
\]
The sizes of all of these sets will be determined later.

Now, we let $\ms S := \ms S_\Pi^\omega$ be the regular subdivision of $\Pi$ with $\omega : \Pi \to \bb R$ defined as follows. First, for every distinct $i$, $j \in [5]$, we set
\begin{alignat*}{3}
\omega(e_i,f) &= 0 \qquad& \omega(e_j,f) &= 1 \qquad& \text{if } f \in \Delta_{i,j}.
\end{alignat*}
Finally, for each $S \in \binom{[5]}{4}$ and distinct $i$, $j$, $k \in S$, let $0 < \epsilon_{S,i,j,k} < 1$ be a generic real number. We set
\begin{alignat*}{3}
\omega(e_i,f) &= 0 \qquad& \omega(e_k,f) &= \epsilon_{S,i,j,k} \qquad&& \text{if } f \in \Delta_{S,i,j,(ik)} \text{ and } k \neq j \\
\omega(e_k,f) &= \epsilon_{S,i,j,k} \qquad& \omega(e_j,f) &= 1 \qquad&& \text{if } f \in \Delta_{S,i,j,(kj)} \text{ and } k \neq i \\
\omega(e_k,f) &= \epsilon_{S,i,j,k} \qquad& \omega(e_{k'},f) &= \epsilon_{S,i,j,k'} \qquad&& \text{if } f \in \Delta_{S,i,j,(kk')} \text{ and } k,k' \neq i,j.
\end{alignat*}

We analyze the cells of $\ms S$. For each $S \in \binom{[5]}{4}$ and distinct $i$, $j \in S$, let
\[
\Pi_{S,i,j} := \bigcup_{(i'j') \in \Gamma_S^2} \bigcup_{f \in \Delta_{S,i,j,(i'j')}} \{(e_{i'},f),(e_{j'},f)\}
\]
and
\[
P_{S,i,j} := \bigcup_{f \in \Delta_{S,i,j,(ij)}} \{(e_i,f),(e_j,f)\}.
\]
In addition, for each $k \in S \minus \{i,j\}$, let
\[
\Xi_{S,i,j,k} := \bigcup_{(i'j') \in \Gamma_{\{i,j,k\}}^2} \bigcup_{f \in \Delta_{S,i,j,(i'j')}} \{(e_{i'},f),(e_{j'},f)\}.
\]

\begin{prop} \label{tricells}
Every cell of $\ms S$ is of the form $C \cup D$, where $D$ is a simplex affinely independent to $C$ and $C$ is a face of one of the following.
\begin{enumerate} \renewcommand{\labelenumi}{(\alph{enumi})}
\item $\Pi_{S,i,j} \cup \Xi_{S',i,j,k} \cup P_{S'',i,j}$ where $S$, $S'$, $S'' \in \binom{[5]}{4}$ are distinct and contain $i$, $j$ and $\{k\} = [5] \minus S$.
\item $\Xi_{S,i,j,k} \cup \Xi_{S',i,j,k'} \cup \Xi_{S'',i,j,k''}$ where $S$, $S'$, $S'' \in \binom{[5]}{4}$ are distinct and contain $i$, $j$ and $k \in S \minus \{i,j\}$, $k' \in S' \minus \{i,j\}$, and $k'' \in S'' \minus \{i,j\}$ are distinct.
\end{enumerate}
We will call the cells in (a) and (b) the \emph{complex} cells of $\ms S$.
\end{prop}

\begin{proof}
Using the notation of Theorem~\ref{fans}, the cell in (a) corresponds to $x \in \bb R^5$ with $x_i = 0$, $x_j = 1$, $x_l = \epsilon_{S,i,j,l}$ for each $l \in S \minus \{i,j\}$, and $x_k = \epsilon_{S',i,j,k}$. The cell in (b) corresponds to $x \in \bb R^5$ with $x_i = 0$, $x_j = 1$, $x_k = \epsilon_{S,i,j,k}$, $x_{k'} = \epsilon_{S',i,j,k'}$, and $x_{k''} = \epsilon_{S'',i,j,k''}$. Due to the genericness of the $\epsilon_{S,i,j,k}$, it can be checked that every cell of $\ms S$ is the union of a face of one of these cells and an affinely independent simplex; we leave the details to the reader.
\end{proof}

Thus, in order to give a triangulation of $\Pi$ which refines $\ms S$, it suffices to specify triangulations of the complex cells of $\ms S$ which agree on common faces. To do this, we first specify triangulations of each $\Pi_{S,i,j}$. We will then use a ``pseudoproduct'' operation to extend these to triangulations of the complex cells.

Fix $S \in \binom{[5]}{4}$ and distinct $i$, $j \in S$. Let $\tilde\Pi$ be the large 3-permutohedron defined in equation \eqref{P3}. Let $\psi : [4] \to S$ be a map such that $\psi(1) = i$ and $\psi(2) = j$. We now choose the sizes of the $\Delta_{S,i,j,(i'j')}$ so that we have an affine isomorphism $\Psi : \tilde\Pi \to \Pi_{S,i,j}$ given by $e_k \mapsto e_{\psi(k)}$ for all $k \in [4]$ and so that $\{f_{kl}^r\}_{-N \le r \le N}$ maps bijectively to $\Delta_{S,i,j,(\psi(k)\psi(l))}$. We will define $\Psi : \tilde\Pi \to \Pi_{S,i,j}$ so that $f_{12}^{-N}$ maps to $f_{i,j}^\ast$.

Let $\tilde{\ms T}$ and $\tilde{\ms C}$ be the triangulation of $\tilde\Pi$ and the collection of triangulations, respectively, constructed in Section~\ref{zonoconstruction} which satisfy Proposition~\ref{probably}. Let $\ms T_{S,i,j}$ and $\tilde{\ms C}_{S,i,j}$ be the images of $\tilde{\ms T}$ and $\tilde{\ms C}$, respectively, under $\Psi$. We thus have a triangulation $\ms T_{S,i,j}$ of each $\Pi_{S,i,j}$.

To extend the $\ms T_{S,i,j}$ to triangulations of the complex cells, we use the following ``ordered pseudoproduct'' construction. It is proved in the Appendix.

\begin{prop} \label{glue}
Let $\Pi_1$, $\Pi_2$, \dots, $\Pi_N$, and $\rho$ be cells of $\Pi$ such that $\rho = \{(e_i,f),(e_j,f)\}$ for some distinct $i$, $j \in [5]$ and $f \in \Delta_{(ij)}$, and for any $1 \le r,s \le N$, we have $G(\Pi_r) \cap G(\Pi_s) = G(\sigma)$. Let $\ms T_1$, \dots, $\ms T_N$ be triangulations of $\Pi_1$, \dots, $\Pi_N$, respectively. Let $\ms M$ be the set of all simplices $\sigma \in A$ of the following form: There is an integer $1 \le s \le N$ such that
\[
\sigma = \left( \bigcup_{r < s} (\sigma_r \minus \rho) \right) \cup \sigma_s \cup \left( \bigcup_{r > s} \sigma_r \right)
\]
where
\begin{itemize}
\item If $r < s$, then $\sigma_r$ is a maximal simplex of $\ms T_r$ with $\rho \subseteq \sigma_r$.
\item $\sigma_s$ is a maximal simplex of $\ms T_s$, and if $s < N$, then $\rho' := \sigma_s \cap \rho \neq \rho$.
\item If $r > s$, then $\sigma_r$ is a maximal simplex of $\ms T_r[F_r]$, where $F_r$ is a facet of $\Pi_r$ with $F_r \cap \rho = \rho'$.
\end{itemize}
Then $\ms M$ is the set of maximal simplices of a triangulation $\ms T(\ms T_1,\dotsb,\ms T_N)$ of $\Pi_1 \cup \dotsb \cup \Pi_N$, and $\ms T_r \subseteq \ms T(\ms T_1,\dotsb,\ms T_N)$ for all $1 \le r \le N$.
\end{prop}

Now, fix distinct $i$, $j \in [5]$, and specify an ordering $S_1$, $S_2$, $S_3$ of the sets $S \in \binom{[5]}{4}$ which contain $i$ and $j$. Each complex cell is of the form $F_1 \cup F_2 \cup F_3$ where for each distinct $1 \le r,s \le 3$, $F_r$ is a face of $\Pi_{S_r,i,j}$, and $G(F_r) \cap G(F_s) = G(\rho)$ where
\[
\rho := \{(e_i,f_{i,j}^\ast),(e_j,f_{i,j}^\ast)\}.
\]
By Proposition~\ref{glue}, we thus have a triangulation $\ms T( \ms T_{S_1,i,j}[F_1], \ms T_{S_2,i,j}[F_2], \ms T_{S_3,i,j}[F_3] )$ of $F_1 \cup F_2 \cup F_3$. We leave it as an exercise to verify that all of these triangulations of the complex cells agree on common faces.\footnote{The argument is the same as the proof of Property 1 in the proof Proposition~\ref{glue}.} Hence, we can extend these triangulations of the complex cells to a triangulation $\ms T$ of $\Pi$. By Proposition~\ref{glue}, this triangulation $\ms T$ contains all the triangulations $\ms T_{S,i,j}$ as subcollections.

For each $S \in \binom{[5]}{4}$ and distinct $i$, $j \in [n]$, let $\ms D_{S,i,j}$ be the set of all zonotopal triangulations of the form $\ms T_{ijk}^{f_1f_2f_3}$ such that $k \in S \minus \{i,j\}$ and $\ms T_{ijk}^{f_{i,j}^\ast f_2f_3} \in \tilde{\ms C}_{S,i,j}$. Let
\[
\ms C_{S,i,j} := \tilde{\ms C}_{S,i,j} \cup \ms D_{S,i,j}.
\]
We now finally show that the assumptions of Lemma~\ref{ensemble2} hold for $\ms C_{S,i,j}$ and $\ms T$.

\begin{prop}
The collections $\ms C_{S,i,j}$ satisfy Properties 1--3 of Lemma~\ref{ensemble2}.
\end{prop}

\begin{proof}
We first prove that Properties 1 and 2 hold. Suppose $\ms T_{i_1i_2i_3}^{f_1f_2f_3} \in \ms C_{S,i,j}$. If $\ms T_{i_1i_2i_3}^{f_1f_2f_3} \in \tilde{\ms C}_{S,i,j}$, then Properties 1 and 2 follow from Proposition~\ref{probably}(B) by setting the appropriate values of $\gamma$. So we may assume $\ms T_{i_1i_2i_3}^{f_1f_2f_3} = \ms T_{ijk}^{f_1f_2f_3} \in \ms D_{S,i,j}$. By definition of $\ms D_{S,i,j}$, we have $\ms T_{ijk}^{f_{i,j}^\ast f_2f_3} \in \tilde{\ms C}_{S,i,j}$. Hence, by Proposition~\ref{probably}(B), there exist $f_1' \in \Delta_{S,i,j,(il)}$, $f_2' \in \Delta_{S,i,j,(jl)}$, and $f_3' \in \Delta_{S,i,j,(kl)}$, where $l = S \minus \{i,j,k\}$, such that
\[
\ms T_{ijl}^{f_{i,j}^\ast f_2'f_1'}, \ms T_{jkl}^{f_2f_3'f_2'}, \ms T_{kil}^{f_3f_1'f_3'} \in \tilde{\ms C}_{S,i,j}.
\]
By definition, we thus have $\ms T_{ijl}^{f_1f_2'f_1'} \in \ms D_{S,i,j}$. Hence
\[
\ms T_{ijl}^{f_1f_2'f_1'}, \ms T_{jkl}^{f_2f_3'f_2'}, \ms T_{kil}^{f_3f_1'f_3'} \in \ms C_{S,i,j}
\]
which proves Property 1. The argument for Property 2 is analogous.

We now prove Property 3. Let $i$, $j$, $l \in [5]$ be distinct and let $f_1 \in \Delta_{(ij)}$. Let $S = [5] \minus \{l\}$. By Proposition~\ref{probably}(A) applied to the triangulation $\ms T_{S,i,j}$, for any $k \in S \minus \{i,j\}$ there exists $f_2 \in \Delta_{S,i,j,(jk)}$ and $f_3 \in \Delta_{S,i,j,(ki)}$ such that
\[
\ms T_{ijk}^{f_{i,j}^\ast f_2 f_3} \in \tilde{\ms C}_{S,i,j}.
\]
Thus by definition, $\ms T_{ijk}^{f_1f_2f_3} \in \ms D_{S,i,j} \subseteq \ms C_{S,i,j}$, which proves Property 3.
\end{proof}

\begin{prop} \label{TcontainsCSij}
Properties \ref{e2contains} and \ref{e2Contains} of Lemma~\ref{ensemble2} hold for $\ms T$ with respect to the collections $\ms C_{S,i,j}$.
\end{prop}

\begin{proof}
We first note the following two facts.

\begin{prop} \label{fstarmax}
For any $f \in \Delta_{(ij)}$, we have $\{ (e_i, f_{i,j}^\ast), (e_j, f) \} \in \ms T$.
\end{prop}

\begin{proof}
In our construction of $\ms S$ from $\omega$, we had
\[
\omega(e_i,f_{i,j}^\ast) - \omega(e_j,f_{i,j}^\ast) + \omega(e_j,f) - \omega(e_i,f) \le 0
\]
for all $f \in \Delta_{(ij)}$, with equality if and only if $f \in \Delta_{i,j}$. In addition, $\tilde{\ms T}$ is a refinement of a regular subdivision of $\tilde{\Pi}$ given by height function $\tilde \omega$ where
\[
\tilde\omega(e_1,f_{12}^{-N}) - \tilde\omega(e_2,f_{12}^{-N}) + \tilde\omega(e_2,f_{12}^r) - \tilde\omega(e_1,f_{12}^r) < 0
\]
for all $r \neq N$. Thus, for all $f \in \Delta_{(ij)}$, restricting $\ms T$ to the face $X_{ij}^{f_{i,j}^\ast f}$ of $\Pi$ yields the triangulation $\ms T_{X_{ij}^{f_{i,j}^\ast f}}^-$, and hence $\{ (e_i, f_{i,j}^\ast), (e_j, f) \} \in \ms T$.
\end{proof}

\begin{prop} \label{downshift}
Let $\ms T_{ijk}^{f_{i,j}^\ast f_2f_3} \subseteq \ms T$ and $f_1 \in \Delta_{(ij)}$. Then $\ms T_{ijk}^{f_1f_2f_3} \subseteq \ms T$.
\end{prop}

\begin{proof}
Let
\[
\sigma_0 := X_{ijk}^{f_{i,j}^\ast f_2f_3} \minus \{(e_i,f_{i,j}^\ast)\} \in \ms T.
\]
By considering a maximal simplex of $\ms T$ containing $\sigma_0$, we have $\sigma_0 \cup \{(e,f_1)\} \in \ms T$ for some $e \in \{e_i,e_j\}$. If $e = e_i$, then $\{(e_i,f_1),(e_j,f_{i,j}^\ast)\} \subseteq \sigma_0$. However, this contradicts Proposition~\ref{fstarmax} and Proposition \ref{circuitparts} in the Appendix. Thus $e = e_j$. Hence, we have
\[
\sigma_1 := X_{ijk}^{f_1f_2f_3} \minus \{(e_i,f_1)\} \in \ms T.
\]
Now, the circuit $X := X_{ijk}^{f_1f_2f_3}$ is a face of $\Pi$, so $\ms T[X]$ is a triangulation of $X$. Since $\sigma_1 \in \ms T[X]$, this triangulation must be $\ms T_X^+$. Thus $\ms T_{ijk}^{f_1f_2f_3} \subseteq \ms T$.
\end{proof}

Now, for any $\ms T_{i_1i_2i_3}^{f_1f_2f_3} \in \tilde{\ms C}_{S,i,j}$, by definition of $\tilde{\ms T}$ and $\tilde{\ms C}$ we have $\ms T_{i_1i_2i_3}^{f_1f_2f_3} \subseteq \ms T_{S,i,j} \subseteq \ms T$. Suppose $\ms T_{ijk}^{f_1f_2f_3} \in \ms D_{S,i,j}$. Then $\ms T_{ijk}^{f_{i,j}^\ast f_2f_3} \in \tilde{\ms C}_{S,i,j}$ by definition. Thus $\ms T_{ijk}^{f_{i,j}^\ast f_2f_3} \subseteq \ms T$, so by Proposition~\ref{downshift}, $\ms T_{ijk}^{f_1f_2f_3} \subseteq \ms T$. Hence Property~\ref{e2contains} holds.

Now suppose $\ms T_{ijk}^{f_1f_2f_3} \in \ms C_{S,i,j}$, and let $f \in \Delta_{(il)}$ where $\{l\} = [5] \minus S$. By Property (i), we have $X_{ijk}^{f_1f_2f_3} \minus \{(e_i,f_1)\} \in \ms T$. By Proposition~\ref{grow} in the Appendix, we have $X_{ijk}^{f_1f_2f_3} \minus \{(e_i,f_1)\} \cup \{(e_i,f)\} \in \ms T$. Thus Property~\ref{e2Contains} holds.
\end{proof}

Hence, we have collections $\ms C_{S,i,j}$ which satisfy the hypotheses of Lemma~\ref{ensemble2}, and by Propositions~\ref{TcontainsCSij} and \ref{extensiontoA}, there exists a triangulation $\ms T' \in \mc S_{\ms C}$. Clearly $\mc S_{\ms C}$ is not the set of all triangulations of $A$; for example, there exist triangulations of $A$ which do not contain any triangulations of circuits of size six \cite{DRS}. Hence Lemma~\ref{ensemble2} proves Theorem~\ref{main2}.

\section{Appendix}

\subsection{Proofs of Propositions \ref{zonoshift} and \ref{trishift}}

We prove the following general fact.

\begin{prop} \label{genshift}
Let $A$ be a subset of $\Delta^{m-1} \times \Delta^{n-1}$, and let $\ms T$ be a triangulation of $A$. Let $\ms T'$ be the result of a flip on $\ms T$ supported on $X = (X^+,X^-)$. Suppose that $\sigma \in \ms T$ and $\sigma \notin \ms T'$, and $G(\sigma)$ is connected. Then $\sigma$ contains a maximal simplex of $\ms T_X^+$.
\end{prop}

\begin{proof}
Let $\tau$ be a maximal simplex of $\ms T$ containing $\sigma$. We must have $\tau \notin \ms T'$ since $\sigma \notin \ms T'$. So $\tau$ must contain a maximal simplex $\tau'$ of $\ms T_X^+$. Moreover, since $\tau \minus \{x\} \in \ms T'$ for any $x \in X^-$, we must have $\sigma \not\subseteq \tau \minus \{x\}$ for any $x \in X^-$. Hence $\sigma \supseteq X^-$. Since $G(\sigma)$ is connected and $G(\sigma \cup \tau') \subseteq G(\tau)$ is acyclic, this can only happen $\sigma \supseteq \tau'$, as desired.
\end{proof}

Proposition~\ref{trishift} now follows immediately. Proposition~\ref{zonoshift} follows by applying Proposition~\ref{genshift} to each maximal simplex $\sigma$ of $\ms T_{ijk}^{rst}$ and noting that $\sigma$ cannot contain a maximal simplex of a circuit $X$ of $\Pi$ except when $X = X_{ijk}^{rst}$. This is because for a zonotope $A$, each element of $\Delta^{n-1}$ has only two neighbors in $G(A)$.

\subsection{Proofs of Propositions \ref{zonoblock} and \ref{triblock}}

Throughout this section, fix a subset $A \subseteq \Delta^{m-1} \times \Delta^{n-1}$. We need the following two facts.

\begin{prop} \label{grow}
Let $\ms T$ be a triangulation of $A$. Let $\sigma \in \ms T$. Then there exists $\tau \in \ms T$ such that $\sigma \subseteq \tau$, and if $f \in \Delta^{n-1}$ is adjacent to $G(\sigma)$ in $G(A)$, then $f$ is adjacent to $G(\sigma)$ in $G(\tau)$.
\end{prop}

\begin{proof}
Let $\Delta_\sigma := \Delta^{m-1} \cap G(\sigma)$. Choose a weak ordering $\le_w$ of $\Delta^{m-1}$ such that 
\[
\min(\le_w,\Delta^{m-1}) = \Delta_\sigma.
\]
Let $F$ be the face of $A$ induced by $\le_w$. Then $\sigma \subseteq F$. Let $\tau$ be a maximal simplex of $\ms T[F]$ containing $\sigma$. Then $\tau$ satisfies the conclusions of the Proposition.
\end{proof}

By symmetry of $\Delta^{m-1}$ and $\Delta^{n-1}$, the above Proposition also holds after replacing $f \in \Delta^{n-1}$ with $e \in \Delta^{m-1}$.

\begin{prop} \label{circuitparts}
Let $\ms T$ be a triangulation of $A$, and suppose $X = (X^+,X^-)$ is a circuit of $A$. Then $\ms T$ does not contain both $X^+$ and $X^-$.
\end{prop}

\begin{proof}
The interiors of the convex hulls of opposite parts of a circuit intersect.
\end{proof}

The next Proposition immediately implies Propositions \ref{zonoblock} and \ref{triblock}. We use the notation for circuits and triangulations of circuits described in Section~\ref{sec:tricomponent}.

\begin{prop}
Let $\ms T$ be a triangulation of $A$. Let $i$, $j$, $k$, $l \in [m]$ be distinct , and for each $\alpha \in \Gamma_{\{i,j,k,l\}}^2$, let $f_\alpha \in \Delta^{n-1}$.
Suppose that
\[
\ms T_{jkl}^{f_{(jk)}f_{(kl)}f_{(lj)}}, \ms T_{kil}^{f_{(ki)}f_{(il)}f_{(lk)}} \subseteq \ms T
\]
and
\[
X_{ijl}^{f_{(ij)}f_{(jl)}f_{(li)}} \minus \{ (e_i, f_{(ij)}) \} \in \ms T.
\]
Then $\ms T$ does not have a flip supported on $X_{ijk}^{f_{(ij)}f_{(jk)}f_{(ki)}}$.
\end{prop}

\begin{proof}
Suppose the contrary. In particular, this means
\[
\ms T_{ijk}^{f_{(ij)}f_{(jk)}f_{(ki)}} \subseteq \ms T.
\]
Set $\sigma_0 = X_{ijl}^{f_{(ij)}f_{(jl)}f_{(li)}} \minus \{ (e_i, f_{(ij)}) \}$. By Proposition~\ref{grow}, there exists $\sigma_1 \in \ms T$ such that
\[
\sigma_1 = \sigma_0 \cup \{(e,f_{(jk)})\} \cup \{(e',f_{(kl)})\}
\]
where $e$, $e' \in \{e_i,e_j,e_l\}$. We prove the following claims.

\paragraph{Claim 1:} $e = e_j$. Suppose first that $e = e_i$. Then $\{(e_i,f_{(jk)}),(e_j,f_{(ij)})\} \subseteq \sigma_1$. However, we also have $\{(e_i,f_{(ij)}),(e_j,f_{(jk)})\} \in \ms T_{ijk}^{f_{(ij)}f_{(jk)}f_{(ki)}} \subseteq \ms T$. This contradicts Proposition~\ref{circuitparts}.

Now suppose that $e = e_l$. Then $\{(e_l,f_{(jk)}),(e_j,f_{(jl)})\} \subseteq \sigma_1$. However, we have $\{(e_l,f_{(jl)}),(e_j,f_{(jk)})\} \in \ms T_{jkl}^{f_{(jk)}f_{(kl)}f_{(lj)}} \subseteq \ms T$, again a contradiction. So $e = e_i$.

\paragraph{Claim 2:} $e' = e_l$. Suppose first that $e' = e_i$. Then $\{(e_i,f_{(kl)}),(e_l,f_{(li)})\} \subseteq \sigma_1$. However, we have $\{(e_i,f_{(li)}),(e_l,f_{(kl)})\} \in \ms T_{kil}^{f_{(ki)}f_{(il)}f_{(lk)}} \subseteq \ms T$, a contradiction.

Now suppose that $e' = e_j$. Then $\{(e_j,f_{(kl)}),(e_l,f_{(jl)})\} \subseteq \sigma_1$. However, we have $\{(e_j,f_{(jl)}),(e_l,f_{(kl)})\} \in \ms T_{jkl}^{f_{(jk)}f_{(kl)}f_{(lj)}} \subseteq \ms T$, a contradiction. So $e' = e_l$.
\\

Now, by Proposition~\ref{grow} and the comment afterwards, there exists $\sigma_2 \in \ms T$ such that
\[
\sigma_2 = \sigma_1 \cup \{ (e_k,f) \}
\]
where $f \in \{f_{(ij)},f_{(jk)},f_{(il)},f_{(jl)},f_{(kl)}\}$. We prove the following.

\paragraph{Claim 3:} $f = f_{(jk)}$. The argument goes as follows:

If $f = f_{(ij)}$, then $\{(e_k,f_{(ij)}),(e_j,f_{(jk)})\} \subseteq \sigma_2$ but $\{(e_k,f_{(jk)}),(e_j,f_{(ij)})\} \in \ms T_{ijk}^{f_{(ij)}f_{(jk)}f_{(ki)}}$.

If $f = f_{(il)}$, then $\{(e_k,f_{(il)}),(e_l,f_{(kl)})\} \subseteq \sigma_2$ but $\{(e_k,f_{(kl)}),(e_l,f_{(il)})\} \in \ms T_{kil}^{f_{(ki)}f_{(il)}f_{(lk)}}$.

If $f = f_{(jl)}$, then $\{(e_k,f_{(jl)}),(e_j,f_{(jk)})\} \subseteq \sigma_2$ but $\{(e_k,f_{(jk)}),(e_j,f_{(jl)})\} \in \ms T_{jkl}^{f_{(jk)}f_{(kl)}f_{(lj)}}$.

If $f = f_{(kl)}$, then $(X_{jkl}^{f_{(jk)}f_{(kl)}f_{(lj)}})^+ \subseteq \sigma_2$, but $(X_{jkl}^{f_{(jk)}f_{(kl)}f_{(lj)}})^- \in \ms T_{jkl}^{f_{(jk)}f_{(kl)}f_{(lj)}}$.

Thus $f = f_{(jk)}$.
\\

Finally, by Proposition~\ref{grow} there exists $\sigma_3 \in \ms T$ such that
\[
\sigma_3 = \sigma_2 \cup \{(e,f_{(ki)})\}
\]
where $e \in \{e_i,e_j,e_k,e_l\}$.

\paragraph{Claim 4:} $e = e_i$. The argument goes as follows.

If $e = e_j$, then $\{(e_j,f_{(ki)}),(e_k,f_{(jk)})\} \subseteq \sigma_3$ but $\{(e_j,f_{(jk)}),(e_k,f_{(ki)})\} \in \ms T_{ijk}^{f_{(ij)}f_{(jk)}f_{(ki)}}$.

If $e = e_k$, then $(X_{kil}^{f_{(ki)}f_{(il)}f_{(lk)}})^+ \subseteq \sigma_3$ but $(X_{kil}^{f_{(ki)}f_{(il)}f_{(lk)}})^- \in \ms T_{kil}^{f_{(ki)}f_{(il)}f_{(lk)}}$.

If $e = e_l$, then $\{(e_l,f_{(ki)}),(e_i,f_{(il)})\} \subseteq \sigma_3$ but $\{(e_l,f_{(il)}),(e_i,f_{(ki)})\} \in \ms T_{kil}^{f_{(ki)}f_{(il)}f_{(lk)}}$.

Thus $e = e_i$.
\\

Now, let $\tau$ be a maximal simplex of $\ms T$ containing $\sigma_3$. Then $(X_{ijk}^{f_{(ij)}f_{(jk)}f_{(ki)}})^- \subseteq \tau$ but $\abs{X_{ijk}^{f_{(ij)}f_{(jk)}f_{(ki)}} \cap \tau} = 4$. This contradicts Proposition~\ref{flip} and our assumption that $\ms T$ has a flip supported on $X_{ijk}^{f_{(ij)}f_{(jk)}f_{(ki)}}$. So $\ms T$ does not have such a flip.
\end{proof}

\subsection{Proof of Proposition~\ref{glue}}

Our proof is based on the following.

\begin{thm}[Rambau \cite{Ram97}] \label{combintri}
Let $\ms M$ be a nonempty collection of full-dimensional simplices of a point set $A$. Then $\ms M$ is the set of maximal simplices of a triangulation of $A$ if and only if
\begin{enumerate}
\item There is no circuit $X = (X^+,X^-)$ of $A$ and simplices $\tau$, $\tau' \in \ms M$ such that $X^+ \subseteq \tau$, $X^- \subseteq \tau'$.
\item For any simplex $\tau \in \ms M$ and facet $\sigma$ of $\tau$, either $\sigma$ is contained in a facet of $A$ or there is another $\tau' \in \ms M$, $\tau' \neq \tau$, such that $\sigma \subseteq \tau'$.
\end{enumerate}
\end{thm}

Let $\ms M$ be as in Proposition~\ref{glue}. We prove Properties 1 and 2 of Theorem~\ref{combintri} for $\ms M$.

\begin{proof}[Proof of Property 1]
Suppose that $X = (X^+,X^-)$ is a circuit of $\Pi_1 \cup \dotsb \cup \Pi_N$ and we have $X^+ \subseteq \sigma$, $X^- \subseteq \sigma'$ for some $\sigma$, $\sigma' \in \ms M$. Let
\begin{align*}
\sigma &= \left( \bigcup_{r < s} (\sigma_r \minus \rho) \right) \cup \sigma_s \cup \left( \bigcup_{r > s} \sigma_r \right) \\
\sigma' &= \left( \bigcup_{r < s'} (\sigma'_r \minus \rho) \right) \cup \sigma'_{s'} \cup \left( \bigcup_{r > s'} \sigma'_r \right)
\end{align*}
as in Proposition~\ref{glue}, with analogous definitions for $\sigma'$. Suppose first that $X \subseteq \Pi_r$ for some $1 \le r \le N$. Then we have $X^+ \subseteq \sigma \cap \Pi_r \subseteq \sigma_r$ and similarly $X^- \subseteq \sigma'_r$. This contradicts the fact that $\sigma_r$, $\sigma'_r$ are simplices of a triangulation $\ms T_r$. So we cannot have this case.
It follows that $X \subseteq \Pi_{r_1} \cup \Pi_{r_2}$ for some $1 \le  r_1 <  r_2 \le N$, and we can uniquely write
\[
X = X_{r_1} \cup X_{r_2} \minus \rho 
\]
where $X_{r_1} = (X_{r_1}^+,X_{r_1}^-)$ and $X_{r_2} = (X_{r_2}^+,X_{r_2}^-)$ are circuits in $\Pi_{r_1}$ and $\Pi_{r_2}$ respectively, and
\begin{alignat*}{2}
X_{r_1}^+ &= (X^+ \cap \Pi_{r_1}) \cup \rho_1 \qquad& X_{r_1}^- &= (X^- \cap \Pi_{r_1}) \cup \rho_2\\
X_{r_2}^+ &= (X^+ \cap \Pi_{r_2}) \cup \rho_2 \qquad& X_{r_2}^- &= (X^- \cap \Pi_{r_2}) \cup \rho_1
\end{alignat*}
where $\rho_1$, $\rho_2$ are different one-element subsets of $\rho$.

We claim that $X_{r_1}^+ \subseteq \sigma_{r_1}$. First suppose that $r_1 < s$. Then $(X^+ \cap \Pi_{r_1}) \cup \rho \subseteq \sigma_{r_1}$, so $X_{r_1}^+ \subseteq \sigma_{r_1}$, as desired. So we may assume $r_1 \ge s$, and hence $s < r_2 \le N$. Now, since $s < N$, we have $\sigma \cap \rho = \rho_i$ for either $i = 1$ or 2. If $i = 1$, then $X_{r_1}^+ \subseteq \sigma_{r_1}$ and we are done. Suppose $i = 2$. Then $X_{r_2}^+ \subseteq \sigma_{r_2}$. Since $r_2 > s$, it follows that $X_{r_2}^+$ is contained in a facet $F$ of $\Pi_{r_2}$ with $F \cap \rho = \rho_2$. However, $X_{r_2}^+ \subseteq F$ implies $X_{r_2} \subseteq F$ because $F$ is a face of $\Pi_{r_2}$. This contradicts $F \cap \rho = \rho_2$. Hence we have $X_{r_1}^+ \subseteq \sigma_{r_1}$.

By the same argument, we have $X_{r_1}^- \subseteq \sigma'_{r_1}$. Hence $X_{r_1}^+$, $X_{r_1}^- \in \ms T_{r_1}$, a contradiction. This proves Property 1.
\end{proof}

\begin{proof}[Proof of Property 2]
Let
\[
\sigma = \left( \bigcup_{r < s} (\sigma_r \minus \rho) \right) \cup \sigma_s \cup \left( \bigcup_{r > s} \sigma_r \right)
\]
be an element of $\ms M$ as before. Let $x \in \sigma$, and consider the facet $\sigma \minus \{x\}$. We have the following cases.
\\

\emph{Case 1:} $x \in \sigma_t \minus \rho$ where $t < s$.

\emph{Subcase 1.1:} $\sigma_t \minus \{x\}$ is not contained in a facet of $\Pi_t$. Then there is some $x' \in \Pi_t$ such that $\sigma_t' := \sigma_t \minus \{x\} \cup \{x'\} \in \ms T_t$. Then
\[
\sigma \minus \{x\} \cup \{x'\} = \left( \bigcup_{\substack{ r < s \\ r \neq t }} (\sigma_r \minus \rho) \right) \cup (\sigma_t' \minus \rho) \cup \sigma_s \cup \left( \bigcup_{r > s} \sigma_r \right)
\]
is an element of $\ms M$ containing $\sigma \minus \{x\}$, as desired.

\emph{Subcase 1.2:} $\sigma_t \minus \{x\}$ is contained in a facet $F$ of $\Pi_t$. Let $\phi : \Pi_t \to \bb R$ be a linear functional supporting $F$ on $\Pi_t$. Since $\rho \subseteq \sigma_t \minus \{x\}$, we have $\rho \subseteq F$, hence $\phi(\rho) = \{b\}$ where $b$ is the minimum of $\phi$ on $\Pi_t$. We may assume $b = 0$.\footnote{This is because $\Pi_t$ is contained in an affine subspace not containing the origin.} Now, we can extend $\phi$ to a linear functional $\phi'$ on $\Pi_1 \cup \dotsb \cup \Pi_N$ by setting $\phi'(v) = \phi(v)$ for all $v \in \Pi_t$ and $\phi'(v) = 0$ for all $v \in \Pi_r$, $r \neq t$.\footnote{We can do this because any affine dependence in $\Pi_1 \cup \dotsb \cup \Pi_N$ can be written as a sum of affine dependencies each contained in one of the $\Pi_r$.} Then the face of $\Pi_1 \cup \dotsb \cup \Pi_N$ supported by $\phi'$ contains $\sigma \minus \{x\}$, and this face is proper because it does not contain $\Pi_t$. Thus $\sigma \minus \{x\}$ is contained in a facet of $\Pi_1 \cup \dotsb \cup \Pi_N$, as desired.
\\

\emph{Case 2:} $x \in \rho$.

\emph{Subcase 2.1:} $\sigma \cap \rho \neq \rho$. Then $\sigma \minus \{x\}$ is contained in a facet of $\Pi_1 \cup \dotsb \cup \Pi_N$ because $f$ has no neighbors in $G(\sigma \minus \{x\})$.

\emph{Subcase 2.2:} $\sigma \cap \rho = \rho$. This implies $s = N$ and $x \in \sigma_r$ for all $r$. First, suppose there is some $t$ such that $\sigma_t \minus \{x\}$ is not contained in a facet of $\Pi_t$, and let $t$ be the largest such number. Then there is some $x' \in \Pi_t$ such that $\sigma_t' := \sigma_t \minus \{x\} \cup \{x'\} \in \ms T_t$. Thus
\[
\sigma \minus \{x\} \cup \{x'\} = \left( \bigcup_{r < t} (\sigma_r \minus \rho) \right) \cup \sigma_t' \cup \left( \bigcup_{r > t} (\sigma_r \minus \{x\}) \right)
\]
is an element of $\ms M$ containing $\sigma \minus \{x\}$, as desired.\footnote{In more detail: By definition of $t$, if $r > t$ then $\sigma_r \minus \{x\}$ is contained in a facet $F$ of $\Pi_r$. We must have $F \cap \rho = \rho \minus \{x\}$ because $\sigma_r$ is full-dimensional in $\Pi_r$.}

Now suppose there is no such $t$. Then for all $r$, $\sigma_r \minus \{x\}$ is contained in a facet $F_r$ of $\Pi_r$. We have $x \notin F_r$ since $\sigma_r$ is full dimensional in $\Pi_r$. For each $r$, let $\phi_r : \Pi_r \to \bb R$ be a linear functional supporting $F_r$ on $\Pi_r$. Let $\{y\} = \rho \minus \{x\}$. Then $y \in F_r$ and $x \notin F_r$, so $\phi_r(y) < \phi_r(x)$. By appropriately choosing $\phi_r$, we may assume $\phi_r(y) = 0$ and $\phi_r(x) = 1$ for all $r$. As before, we can then define a linear functional $\phi'$ on $\Pi_1 \cup \dotsb \cup \Pi_N$ such that $\phi'(v) = \phi_r(v)$ for all $v \in \Pi_r$. Then the face of $\Pi_1 \cup \dotsb \cup \Pi_N$ supported by $\phi'$ contains $\sigma \minus \{x\}$, and this face is proper. Thus $\sigma \minus \{x\}$ is contained in a facet of $\Pi_1 \cup \dotsb \cup \Pi_N$.
\\

\emph{Case 3:} $x \in \sigma_s \minus \rho$.

\emph{Subcase 3.1:} $\sigma_s \minus \{x\}$ is not contained in a facet of $\Pi_s$. Then there is some $x' \in \Pi_s$ such that $\sigma'_s := \sigma_s \minus \{x\} \cup \{x'\} \in \ms T_s$. First assume that $x' \notin \rho$. Then
\[
\sigma \minus \{x\} \cup \{x'\} = \left( \bigcup_{r < s} (\sigma_r \minus \rho) \right) \cup \sigma_s' \cup \left( \bigcup_{r > s} \sigma_r \right)
\]
is an element of $\ms M$ containing $\sigma \minus \{x\}$, as desired.

Now assume that $x' \in \rho$. First, suppose that for all $r > s$, $\sigma_r' := \sigma_r \cup \{x'\} \in \ms T_r$. Then
\[
\sigma \minus \{x\} \cup \{x'\} = \left( \bigcup_{r < s} (\sigma_r \minus \rho) \right) \cup (\sigma_s' \minus \rho) \cup \left( \bigcup_{s < r < N} (\sigma_r' \minus \rho) \right) \cup \sigma_N'
\]
is an element of $\ms M$ containing $\sigma \minus \{x\}$, as desired. Now suppose that there is some $t > s$ such that $\sigma_t \cup \{x'\} \notin \ms T_t$, and let $t$ be the smallest such number. Let $x''$ be a point in $\ms T_t$ such that $x'' \notin \sigma_t$ and $\sigma_t'' := \sigma_t \cup \{x''\} \in \ms T_t$. By definition of $t$, $x'' \neq x'$, and hence $x'' \notin \rho$. Then
\[
\sigma \minus \{x\} \cup \{x''\} = \left( \bigcup_{r < s} (\sigma_r \minus \rho) \right) \cup (\sigma_s' \minus \rho) \cup \left( \bigcup_{s < r < t} (\sigma_r' \minus \rho) \right) \cup \sigma_t'' \cup \left( \bigcup_{r > t} \sigma_r \right)
\]
is an element of $\ms M$ containing $\sigma \minus \{x\}$, as desired.

\emph{Subcase 3.2:} $\sigma_s \minus \{x\}$ is contained in a facet $F$ of $\Pi_s$. If $\rho \subseteq F$, then by the same argument as in Subcase 1.2, $\sigma \minus \{x\}$ is contained in a facet of $\Pi_1 \cup \dotsb \cup \Pi_N$. So we may assume $\rho \not\subseteq F$. In particular, this means $\sigma_s \cap \rho = \rho \minus \{x'\}$ for some $x' \in \rho$.

First, suppose that there is some $t < s$ such that $\sigma_t \minus \{x'\}$ is not contained in a facet of $\ms T_t$, and let $t$ be the largest such number. Then there is some $x'' \in \ms T_t$ such that $\sigma_t'' := \sigma_t \minus \{x'\} \cup \{x''\} \in \ms T_t$. Then
\[
\sigma \minus \{x\} \cup \{x''\} = \left( \bigcup_{r < t} (\sigma_r \minus \rho) \right) \cup \sigma_t'' \cup \left( \bigcup_{t < r < s} (\sigma_r \minus \{x'\}) \right) \cup (\sigma_s \minus \{x\}) \cup \left( \bigcup_{r > s} \sigma_r \right)
\]
is an element of $\ms M$ containing $\sigma \minus \{x\}$, as desired.

Now suppose that there is no such $t$. Then for all $r$, we have that $(\sigma \minus \{x\}) \cap \Pi_r$ is contained in a facet $F_r$ of $\Pi_r$. We also have $F_r \cap \rho = \rho \minus \{x'\}$ for all $r$: For $r < s$, this holds because $\sigma_r$ is maximal in $\Pi_r$, for $r = s$, this holds by our original assumption, and for $r > s$, this holds by definition. Hence, by the argument in the second part of Subcase 2.2, $\sigma \minus \{x\}$ is contained in a facet of $\Pi_1 \cup \dotsb \cup \Pi_N$, as desired.
\\

\emph{Case 4:} $x \in \sigma_t \minus \rho$ where $t > s$.

Let $\rho' = \sigma_t \cap \rho$. Let $\ms U \subseteq \ms T_t$ be the simplicial complex consisting of all simplices in $\ms T_t$ on the boundary of $\Pi_t$. Let $\ms U'$ be the subcomplex of $\ms U$ consisting of all simplices in $\ms T_t$ contained in facets $F$ of $\Pi_t$ with $F \cap \rho = \rho'$. In particular, we have $\sigma_t \in \ms U'$.

Since $\ms U$ is homeomorphic to a sphere, there is some $x' \in \Pi_t$ such that $\sigma_t' := \sigma_t \minus \{x\} \cup \{x'\} \in \ms U$. If $\sigma_t' \in \ms U'$, then
\[
\sigma \minus \{x\} \cup \{x'\} = \left( \bigcup_{r < s} (\sigma_r \minus \rho) \right) \cup \sigma_s \cup \sigma_t' \cup \left( \bigcup_{\substack{r > s \\ r \neq t}} \sigma_r \right)
\]
is an element of $\ms M$ containing $\sigma \minus \{x\}$, as desired. Suppose $\sigma_t' \notin \ms U'$. Let $F$ be the facet of $\Pi_t$ containing $\sigma_t'$. Since $\rho' \subseteq \sigma_t'$, but $F \cap \rho \neq \rho'$, we must have $\rho \subseteq F$. By the argument from Subcase 1.2, it follows that $\sigma \minus \{x\}$ is contained in a facet of $\Pi_1 \cup \dotsb \cup \Pi_N$, as desired.
\end{proof}

This proves that $\ms M$ is the set of maximal simplices of a triangulation $\ms T$. Suppose that $\sigma_s$ is a maximal simplex of $\ms T_s$. If $\rho \not\subseteq \sigma_s$, then for all $r \neq s$ we can choose $\sigma_r \in \Pi_r$ such that
\[
\left( \bigcup_{r < s} (\sigma_r \minus \rho) \right) \cup \sigma_s \cup \left( \bigcup_{r > s} \sigma_r \right)
\]
is in $\ms M$. If $\rho \subseteq \sigma_s$, then for all $r \neq s$ we can choose $\sigma_r \in \Pi_r$ such that $\rho \subseteq \sigma_r$ for all $r$ and
\[
\sigma = \left( \bigcup_{r < s} (\sigma_r \minus \rho) \right) \cup (\sigma_s \minus \rho) \cup \left( \bigcup_{s < r < N} (\sigma_r \minus \rho) \right) \cup \sigma_N
\]
is in $\ms M$. Either way, $\sigma_s \in \ms T$. Thus, $\ms T_1$, \dots, $\ms T_N \subseteq \ms T$.


\begin{thebibliography}{99}

\bibitem{BS92} L.J.\ Billera, B.\ Sturmfels, Fiber polytopes, \emph{Ann.\ of Math.} {\bf 135} (1992) 527--549.

\bibitem{BKS94} L.J.\ Billera, M.M.\ Kapranov, B.\ Sturmfels, Cellular strings on polytopes, \emph{Proc. Amer. Math. Soc.} {\bf 122} (1994),
549--555.

\bibitem{DS04} M.\ Develin, B.\ Sturmfels, Tropical convexity, \emph{Doc.\ Math.} {\bf 9} (2004), 1--27.

\bibitem{HRS00} B.\ Huber, J.\ Rambau, F.\ Santos, The Cayley Trick, lifting subdivisions and the Bohne-Dress theorem on zonotopal tilings, \emph{J. Eur. Math. Soc.} {\bf 2} (2000), no.\ 2, 179--198.

\bibitem{DRS} J.\ De Loera, J.\ Rambau, F.\ Santos, \emph{Triangulations: Structures for Algorithms and Applications}, Algorithms and Computation in Mathematics, Vol.\ 25, Springer-Verlag, 2010.

\bibitem{Ken93} R.\ Kenyon, Tiling a polygon with parallelograms, \emph{Algorithmica} {\bf 9} (1993), 382--397. 

\bibitem{Liu16} G.\ Liu, Flip-connectivity of triangulations of the product of a tetrahedron and simplex, preprint, arXiv:1601.06031, 2016.

\bibitem{Ram97} J.\ Rambau, Triangulations of cyclic polytopes and higher Bruhat orders, \emph{Mathematika} {\bf 44} (1997), 162--194.

\bibitem{Rei99} V.\ Reiner, The generalized Baues problem, in: L.J.\ Billera et al., (Eds.), \emph{New Perspectives in Algebraic Combinatorics}, MSRI Book Series, Vol. 38, Cambridge University Press, New York, 1999, pp. 293--336.

\bibitem{San00} F.\ Santos, A point configuration whose space of triangulations is disconnected, \emph{J.\ Amer.\ Math.\ Soc.} {\bf 13} (2000), 611--637.

\bibitem{San02} F.\ Santos, Triangulations of oriented matroids, \emph{Memoirs of the American Mathematical Society} {\bf 156} (2002) No.\ 741.

\bibitem{San05} F.\ Santos, The Cayley trick and triangulations of products of simplices, in \emph{Integer Points in Polyhedra---Geometry, Number Theory, Algebra, Optimization}, ed.\ A.\ Barvinok, M.\ Beck, C.\ Haase, B.\ Reznick, and V.\ Welkner, Contemporary Mathematics {\bf 374}, American Mathematical Society (2005), 151--177.

\bibitem{San06} F.\ Santos, Geometric bistellar flips: the setting, the context and a construction, in \emph{International Congress of Mathematicians, Vol.\ III}, Eur.\ Math.\ Soc., Z\"{u}rich (2006), 931--962.

\bibitem{Stu94} B.\ Sturmfels, On the Newton polytope of the resultant, \emph{J.\ Algebraic Combin.} {\bf 3} (1994), 207--236.

\bibitem{SZ93} B.\ Sturmfels, G.M.\ Ziegler, Extension spaces of oriented matroids, \emph{Discrete Comput.\ Geom.} {\bf 10} (1993), 23--45.

\bibitem{Zie93} G.M.\ Ziegler, Higher Bruhat orders and cyclic hyperplane arrangements, \emph{Topology} {\bf 32} (1993), 259--279.

\end{thebibliography}
\end{document}